\documentclass[11pt,a4paper,reqno]{amsart}
\usepackage{amsmath,amssymb,amsfonts,epsfig,mathrsfs,cite}
\usepackage[T1]{fontenc}
\usepackage{color}
\usepackage{array}
\usepackage{amsthm}
\usepackage{amstext}
\usepackage{graphicx}
\usepackage{setspace}

\makeatletter
\@namedef{subjclassname@2020}{%
  \textup{2020} Mathematics Subject Classification}
\makeatother

\usepackage[margin=2.5cm]{geometry}
\usepackage{color}
\usepackage{enumitem}
\usepackage{amscd,psfrag}
\usepackage{yhmath}
\usepackage[mathscr]{eucal}
\usepackage{comment}

\allowdisplaybreaks[4]

\usepackage{slashed}

\makeatletter
\pdfpageheight\paperheight
\pdfpagewidth\paperwidth

\usepackage{epstopdf}
\usepackage{chngcntr}
\counterwithin{figure}{section}
\usepackage{mathrsfs}

\usepackage{indentfirst}	

\usepackage[normalem]{ulem}
\theoremstyle{plain}
\numberwithin{equation}{section}

\newcommand{\im}{{\sqrt{-1}}}

\newtheorem{definition}{Definition}[section]
\newtheorem{theorem}[definition]{Theorem}
\newtheorem*{theorem*}{Theorem}

\newtheorem{remark}[definition]{Remark}

\newtheorem*{remark*}{Remark}
\newtheorem*{sideremark*}{Side Remark}

\newtheorem*{claim*}{Claim}
\newtheorem*{q*}{Question}
\newtheorem{lemma}[definition]{Lemma}
\newtheorem{corollary}[definition]{Corollary}
\newtheorem*{corollary*}{Corollary}

\newtheorem{proposition}[definition]{Proposition}

\newtheorem{notation}[definition]{Notation}

\newcommand{\red}[1]{\textcolor{red}{#1}}

\newcommand{\CC}{{\mathscr{C}}}
\newcommand{\R}{\mathbb{R}}

\newcommand{\na}{\nabla}
\newcommand{\emb}{\hookrightarrow}

\newcommand{\p}{\partial}
\newcommand{\loc}{{\rm loc}}
\newcommand{\weak}{\rightharpoonup}
\newcommand{\e}{\varepsilon}
\newcommand{\C}{\mathbb{C}}
\newcommand{\dd}{{\rm d}}

\newcommand{\dvg}{{\,\dd {\rm vol}_g^X}}
\newcommand{\G}{\Gamma}

\newcommand{\two}{{\rm II}}
\newcommand{\M}{{\mathcal{M}}}
\newcommand{\bra}{\left\langle}
\newcommand{\ket}{\right\rangle}

\newcommand{\proj}{\mathbf{\Pi}^{\dd}}

\newcommand{\poly}{{\mathscr{P}}}

\newcommand{\sol}{{\mathfrak{S}}}

\newcommand{\so}{{\mathfrak{so}}}
\newcommand{\sonk}{{\so(N+k)}}

\newcommand{\bara}{{\overline{\alpha}}}

\newcommand{\barb}{{\overline{\beta}}}

\newcommand{\ppl}{{\sigma_{\bf ppl}}}

\newcommand{\bart}{{\overline{\theta}}}

\newcommand{\rnk}{{\R^{N+k}}}

\def\XXint#1#2#3{{\setbox0=\hbox{$#1{#2#3}{\int}$ }
\vcenter{\hbox{$#2#3$ }}\kern-.6\wd0}}

\title{A wedge product theorem of compensated compactness with critical exponents on Riemannian manifolds} 

\author{Xiaojin Bai}
\address{Xiaojin Bai: School of Mathematical Sciences, Shanghai Jiao Tong University, No.~6 Science Buildings,
800 Dongchuan Road, Minhang District, Shanghai, China (200240)}
\email{\texttt{baileymoon@sjtu.edu.cn}}

\author{Siran Li}

\address{Siran Li: School of Mathematical Sciences $\&$ CMA-Shanghai, Shanghai Jiao Tong University, No.~6 Science Buildings,
800 Dongchuan Road, Minhang District, Shanghai, China (200240)}

\email{\texttt{siran.li@sjtu.edu.cn}}

\author{Xiangxiang Su}
\address{Xiangxiang Su: School of Mathematical Sciences, Shanghai Jiao Tong University, No.~6 Science Buildings,
800 Dongchuan Road, Minhang District, Shanghai, China (200240)}
\email{\texttt{sjtusxx@sjtu.edu.cn}}


\date{\today}

\begin{document}

\begin{abstract}
We formulate and prove compensated compactness theorems concerning the limiting behaviour of wedge products of weakly convergent differential forms on closed Riemannian manifolds \emph{\`{a} la} Robbin--Rogers--Temple [\textit{Trans. Amer. Math. Soc.} \textbf{303} (1987), 609--618]. The case of critical regularity exponents is considered, which generalises the div-curl lemma in Briane--Casado-D\'{i}az--Murat [\textit{J. Math. Pures Appl.} \textbf{91} (2009), 476--494] for vectorfields, thus going beyond the regularity regime entailed by H\"{o}lder's inequality. Implications on the weak continuity of Gauss--Codazz--Ricci equations and $L^p$-extrinsic geometry of isometric immersions of Riemannian manifolds are discussed.
\end{abstract}
\maketitle

\section{Introduction}\label{sec: intro}

We establish compensated compactness theorems with critical exponents for wedge products of differential forms on closed manifolds. Our work generalises and is, in fact, primarily motivated by the works of Robbin--Rogers--Temple \cite{key1} and Briane--Casado-D\'{i}az--Murat \cite{bcm}. Throughout, by closed manifold we mean a compact Riemannian manifold with no boundary.

Compensated compactness has played an important role in nonlinear analysis, especially in nonlinear PDEs arising from fluid mechanics (\emph{cf}. DiPerna \cite{dip}, C. Dafermos \cite{dafermos}), nonlinear elasticity (\emph{cf}. Re\u{s}etnjak \cite{res}, Ball \cite{ball}, Murat \cite{mur1}, M\"{u}ller \cite{mul}), and geometric analysis (\emph{cf}. H\'{e}lein \cite{h}), as well as calculus of variations (\emph{cf}. Tartar \cite{tar1, tar2}, Ball--Currie--Olver \cite{bco},  Fonseca--Leoni--Mal\'{y} \cite{flm}). The above list of references is by no means exhaustive; we refer to the monograph by Evans \cite{evans} and the recent survey \cite{chen} by Chen for further details.

The \emph{div-curl lemma} introduced by Murat \cite{mur1, mur2, mur3} and Tartar \cite{tar1, tar2} marks a cornerstone of the compensated compactness theory. It ascertains that a nonlinear functional of weakly convergent sequences will converge in coarse topologies when specific derivatives of these sequences are compact in natural function spaces. The nonlinear functional and differential constraints in consideration need to satisfy some conditions that are essentially algebraic in nature. In the simplest and original form, the div-curl lemma reads as follows:
\begin{lemma}\label{lem: div-curl}
Let $\{u^n\}$, $\{v^n\}$ be sequences of vectorfields in $L^2_\loc(\R^3,\R^3)$ such that $u^n \weak \bar{u}$ and $v^n \weak \bar{v}$  in $L^2_\loc(\R^3,\R^3)$. Assume  $\left\{{\rm div}(u^n)\right\}$ and $\left\{{\rm curl}(v^n)\right\}$ are precompact in $W^{-1,2}_\loc(\R^3,\R)$ and  $W^{-1,2}_\loc(\R^3,\R^3)$, respectively. Then $u^n \cdot v^n$ converge to $\bar{u} \cdot \bar{v}$ in the sense of distributions. 
\end{lemma}

This result has been proved, reproved, and generalised since its first appearance. Murat and Tartar's  arguments \cite{mur1, mur2, mur3, tar1, tar2} utilised Fourier transforms. Two harmonic analytic proofs were later provided in the seminal paper by Coifman--Lions--Meyer--Semmes \cite{clms}, which furthermore established that $\left\{u^n \cdot v^n\right\}$ in Lemma~\ref{lem: div-curl} is bounded in the Hardy space $\mathcal{H}^1_\loc$. From another perspective, Robbin--Rogers--Temple \cite{key1} first observed that, by writing 
\begin{align*}
u^n = \Delta \Delta^{-1}u^n = \left({\rm grad} \circ {\rm div} - {\rm curl} \circ {\rm curl} \right)\Delta^{-1}u^n,
\end{align*}
one may decompose $u^n$, and similarly for $v^n$, into a weakly convergent part and a strongly convergent part. The pairing of the two weakly convergent parts is shown to converge in the sense of distributions by employing the ellipticity of $\Delta$, the commutativity of $\Delta^{-1}$ with divergence and curl, as well as the first-order differential constraints in the assumption\footnote{For simplicity of presentations, here we do not state carefully the boundary conditions.}. This idea has been exploited to extend the div-curl lemma to more general domains (\emph{e.g.}, Riemannian manifolds) and differential operators (\emph{e.g.}, exterior differential $\dd$ and codifferential $\dd^*$, or elliptic complexes in greater generality), based on which a functional analytic framework for div-curl lemma has been developed. See Kozono--Yanagisawa \cite{ky}, Chen--Li \cite{cl1, cl2}, Waurick \cite{wau}, and Pauly \cite{pau}. %We also refer to Tartar \cite{tar3} for some historical perspectives on the div-curl lemma.

In fact, Robbin--Rogers--Temple established a more general result \cite[Theorem~1.1]{key1} than Lemma~\ref{lem: div-curl}, formulated in terms of wedge products of weakly convergent differential forms. 

\begin{lemma}[(Multilinear) wedge product theorem; Theorem~1.1 in \cite{key1}]\label{lemma: RRT}
Let $\{\alpha_1^n, \ldots, \alpha_L^n\}_{n \in \mathbb{N}}$ be sequences of differential forms on $\Omega \subset \R^N$ of degrees $s_1, \ldots, s_L$, respectively; $\sum_{i=1}^L s_i \leq N$. Assume for each $i \in \{1,\ldots,L\}$ that
\begin{align*}
\alpha_i^n \weak \overline{\alpha_i} \text{ in } L^{p_i} \text{ with } \sum_{i=1}^L\frac{1}{p_i}=1
\end{align*}
and that
\begin{align*}
\left\{\dd \alpha_i^n\right\} \text{ lies in a compact subset of } W^{-1,p_i}_\loc.
\end{align*}
Then we have
\begin{align*}
\alpha_1^n \wedge \cdots \wedge \alpha_L^n \longrightarrow \overline{\alpha}_1 \wedge \cdots \wedge \overline{\alpha}_L \qquad \text{ in } \mathcal{D}'.
\end{align*}
\end{lemma}

In terms of applications, we find that such formulation is particularly convenient for weak continuity considerations of isometric immersions of (semi-)Riemannian manifolds. See \cite{cl1, cl2}, in which some generalisations of Lemma~\ref{lemma: RRT} are proved and exploited. This observation has been further extended to study weak continuity  of gauge equations \cite{cg, ls}.

The theory of compensated compactness, as well as its descendants and relatives of the div-curl lemma, remain an important topic of research in nonlinear analysis today. Some recent endeavours in this field are devoted to developing nonhomogeneous, endpoint, and fractional versions of div-curl lemmas (see \cite{daf, ls, bc, bcm, ms, brezis, grs}), refined characterisations of the structures --- especially, algebraic structures --- of compensated compactness quantities (see \cite{mur3, lmzw, zhou, io, mm, rin1}), and connections with broader contexts in geometric measure theory (see, \emph{e.g.}, \cite{sil}). The above list of references is by no means exhaustive.

\emph{Trente ans après} the introduction of the div-curl lemma by Murat--Tartar \cite{mur1, tar1}, Briane--Casado-D\'{i}az--Murat further proved in \cite{bcm} several endpoint div-curl-type theorems and applied them to study the $G$-convergence for unbounded monotone operators. The setting is as follows: Let $\{u^n\} \subset L^p$ and $\{v^n\}\subset L^q$ be two sequences of weakly convergent vectorfields on $\Omega \subset \R^N$ to limits $\bar{u}$ and $\bar{v}$, respectively, such that $|u^n - \bar{u}|^p$ and $|v^n - \bar{v}|^q$ both converge weakly-$\star$ as Radon measures. In addition, assume the strong convergence of ${\rm div} (u^n)$ and ${\rm curl} (v^n)$ in $W^{-1,q'}$ and $W^{-1,p'}$, respectively.\footnote{Note the ``twist'' of indices $p$, $q$ here.} Then the dot products $u^n \cdot v^n$ converge to $\bar{u} \cdot \bar{v} + \varpi$ in the sense of distributions, where $\varpi$ is a defect Radon measure that can be characterised in detail. A crucial point here is that the range of indices $(p,q)$ is $\frac{1}{p} + \frac{1}{q} \leq 1+\frac{1}{N}$, which goes beyond the ``subcritical'' case of $(p,q)$ entailed by H\"{o}lder's inequality as in Lemma~\ref{lemma: RRT}. This motivates us to generalise the wedge product theorem to ``critical'' indices.

Our main Theorem~\ref{thm: wedge} below extends Lemma~\ref{lemma: RRT}, the wedge product theorem of compensated compactness \emph{\`{a} la} Robbin--Rogers--Temple \cite{key1}, in the spirit of Briane--Casado-D\'{i}az--Murat \cite{bcm}. Here we focus on the bilinear case $L=2$. See Theorem~\ref{thm: multiple wedge prod} below for an extension for general $L$, namely the multilinear case. Meanwhile, Theorem~\ref{thm: wedge} also generalises \cite{bcm}, by way of extending the results therein from vectorfields on Euclidean domains to differential forms of arbitrary degree on closed manifolds.

Before stating the theorem, we first introduce some notations used throughout this work.
 
 \begin{notation}
For a closed (\emph{i.e.}, compact and boundaryless) smooth Riemannian manifold $X$, denote by $\dvg$ the Riemannian volume measure on $X$, and by $\M(X)$ the space of Radon measures on $X$. In view of the Riesz representation theorem, $\M(X)$ equipped with the total variation norm is isometrically isomorphic to $\left[\CC^0(X)\right]^*$ as a Banach space, topologised by the weak-$\star$ topology on $\left[\CC^0(X)\right]^*$. Write the limit of a sequence of Radon measures as $\M-\lim_{n \to \infty}$.

We write $\dd$ ($\wedge$, resp.) for the exterior differential (wedge product, resp.) on the space of differential forms, and $\dd^*$ for the codifferential, namely that the formal $L^2$-adjoint of $\dd$ taken with respect to the Riemannian metric. We reserve the symbol $\delta$ for Dirac delta measures. Also denote $\Delta:=\dd\dd^*+\dd^*\dd$, which is the Laplace--Beltrami operator. It differs from the Hodge Laplacian by a sign; for example, $\Delta = - \sum_{j=1}^N \p_{jj}$ on $\R^N$.

For a regularity scale $\mathcal{R} = W^{k,p}, \CC^{0,\alpha}, \CC^\infty, \M, \mathcal{D}' \cdots$ (as usual, $\mathcal{D}'$ denotes distributions), we shall designate $\mathcal{R}\left(X, \bigwedge^\ell T^*X\right)$ for the space of differential $\ell$-forms of regularity $\mathcal{R}$ on $X$. Lengths/moduli and inner products on $\bigwedge^\bullet T^*X$, unless specified otherwise, are always taken with respect to the Riemannian metric on $X$. A differential form $\alpha \in \mathcal{R}\left(X, \bigwedge^\bullet T^*X\right)$ is said to be exact (coexact, resp.) if and only if $\dd\alpha=0$ ($\dd^*\alpha=0$, resp.) in the sense of distributions.

We write $A \lesssim_{c_1, \ldots, c_n} B$ to mean the inequality $A \leq CB$, where the constant $C$ depends only on the parameters $c_1, \ldots, c_n$. Einstein's summation convention is adopted throughout.

\end{notation}

Now we may state our generalised bilinear wedge product theorem:
\begin{theorem}\label{thm: wedge}
Let $\left(X,g\right)$ be an $N$-dimensional closed Riemannian manifold; $N \geq 2$. Let $1<p,q<\infty$ be such that
\begin{equation}\label{p,q condition}
1 \leq \frac{1}{p} + \frac{1}{q} \leq 1+\frac{1}{N}.
\end{equation}
Consider sequences of differential forms $\left\{ \alpha^n \right\}\subset L^p\left(X, \bigwedge^{\ell_1} T^*X \right)$ and $\left\{ \beta^n \right\}\subset L^q\left(X, \bigwedge^{\ell_2} T^*X \right)$ satisfying the following conditions:
\begin{enumerate}
\item
$\alpha^n \weak \bara$ weakly in $L^p$;
\item
$\beta^n \weak \barb$ weakly in $L^q$;
\item
$\left| \alpha^n - \bara \right|^p\dvg \weak \mu$ weakly-$\star$ in $\M$;
\item
$\left| \beta^n - \barb \right|^q\dvg \weak \nu$ weakly-$\star$ in $\M$;
\item
$\dd\alpha^n \to \dd\bara$ strongly in $W^{-1,q'}$;
\item
$\dd \beta^n \to \dd\barb$ strongly in $W^{-1,p'}$. 
\end{enumerate}
Then we have the following convergence (modulo subsequences) in the sense of  distributions:
\begin{equation}\label{conclusion of wedge prod thm}
\alpha^n \wedge \beta^n \longrightarrow  \bara \wedge \barb + \sum_{k=1}^\infty \dd \left(v^k\delta_{x^k}\right) \qquad \text{in } \mathcal{D}',
\end{equation}
for some sequences of points $\left\{x^k\right\} \subset X$ and $(\ell_1+\ell_2-1)$-covectors $\left\{v^k\right\}$. Moreover,
\begin{align*}
\left|v^k\right| \lesssim_{p,q,N,X} \left[\mu\left(\left\{x^k\right\}\right)\right]^{\frac{1}{p}}\left[\nu\left(\left\{x^k\right\}\right)\right]^{\frac{1}{q}}.
\end{align*} 

If, in addition, $\frac{1}{p}+\frac{1}{q}<1+\frac{1}{N}$ in \eqref{p,q condition}, then all $v^k$ are zero.
 \end{theorem}

 We shall refer to $\sum_{k=1}^\infty \dd \left(v^k\delta_{x^k}\right)$ in \eqref{conclusion of wedge prod thm} as the \emph{concentration distribution} in the sequel. The indices $(p,q)$ satisfying \eqref{p,q condition} are said to be in the \emph{critical} range.

\begin{remark}
More precisely, \eqref{conclusion of wedge prod thm} means the following:
\begin{align*}
\bra\alpha^n \wedge \beta^n, \Xi\ket \longrightarrow  \bra \bara \wedge \barb, \Xi\ket + \sum_{k=1}^\infty \bra v^k, \dd^*\Xi \ket \delta_{x^k} \qquad \text{for any } \Xi \in \CC^\infty\left(X;\bigwedge^{\ell_1+\ell_2} T^*X\right).
\end{align*}
The pairings $\bra\bullet,\bullet\ket$, as before, are taken with respect to the metric $g$. In geometric measure theoretic terminologies,  the convergence in \eqref{conclusion of wedge prod thm} should be understood as weak convergence of $(\ell_1+\ell_2)$-dimensional currents on $(X,g)$. The same convention is adopted throughout this paper. 
\end{remark}

If $\frac{1}{p}+\frac{1}{q}< 1$ is assumed instead of \eqref{p,q condition}, Theorem~\ref{thm: wedge} is then reduced to the classical wedge product theorem of Robbin--Rogers--Temple
(Lemma~\ref{lemma: RRT} above with $L=2$). Thus, throughout this paper we  focus solely on the range of indices  in \eqref{p,q condition}. In this case, however, notice below:

\begin{remark}\label{rem: low reg wedge}
The definition of the  wedge products $\alpha^n \wedge \beta^n$ and $\bara \wedge \barb$ requires further  clarification: \emph{a priori} it is unclear if they are well defined distributions for $p^{-1}+q^{-1}>1$. We resort to Hodge decomposition for this purpose. See \S\ref{subsec: weak weak pairing} below for details.
\end{remark}

The proof of Theorem~\ref{thm: wedge} occupies \S\ref{sec: proof of wedge}. A crucial assumption in this theorem is that $1<p,q<\infty$, which shall be referred to as the non-endpoint case --- in contrast to the \emph{endpoint} case tackled in \S\ref{sec: endpoint critical case}, \emph{i.e.}, either $p=1$ or $q=1$. Our strategies of the proof follow  Briane--Casado-D\'{i}az--Murat \cite{bcm}, with the key tool for tackling critical exponents being P.-L. Lions' theory of concentrated compactness  \cite{lions}.

Next, in \S\ref{sec: consequences of wedge prod thm}, we discuss several corollaries of the wedge product Theorem~\ref{thm: wedge}. It will be explained that  Theorem~\ref{thm: wedge} encompasses the classical div-curl lemma of Murat--Tartar \cite{mur1, mur2, tar1, tar2} (\emph{cf}. Coifman--Lions--Meyer--Semmes \cite{clms} too), as well as the generalisation by Briane--Casado-D\'{i}az--Murat \cite{bcm}. We also present a multilinear version of the wedge product theorem, Theorem~\ref{thm: multiple wedge prod}, which addresses the distributional convergence of $\alpha_1^n \wedge \cdots \wedge \alpha_L^n$ for sequences of differential forms $\{\alpha^n_i\}_{i \in \{1,2,\ldots,L\};\,n \in \mathbb{N}}$ of $L^{p_i}$-regularity, respectively, provided that $1\leq \sum_{i=1}^L \frac{1}{p_i} \leq 1+\frac{1}{\dim X}$. Consequences of the exactness  of the concentration measure will be discussed too.

The following section \S\ref{sec: endpoint critical case} deals with the ``endpoint critical case'' of Theorem~\ref{thm: wedge}. By ``critical'' we mean $\frac{1}{p} + \frac{1}{q} = 1+\frac{1}{N}$, and by ``endpoint'' we further mean that either $p$ or $q$ equals $1$. The main result in this case, which generalises \cite[\S\S 3$\&$4]{bcm}, is as follows:

\begin{theorem}\label{thm: endpoint}
Let $\left(X^N,g\right)$ be an $N$-dimensional closed Riemannian manifold; $N \geq 2$. 
Consider two sequences of differential forms $\left\{ \alpha^n \right\}\subset \M\left(X, \bigwedge^{\ell_1} T^*X \right)$ and $\left\{ \beta^n \right\}\subset L^N\left(X, \bigwedge^{\ell_2} T^*X \right)$ satisfying the following conditions:
\begin{enumerate}
\item
$\alpha^n \weak \bara$ weakly-$\star$ in $\M$;
\item
$\beta^n \weak \barb$ weakly in $L^N$;
\item
$\left| \alpha^n - \bara \right| \weak \mu$ weakly-$\star$ in $\M$;
\item
$\left| \beta^n - \barb \right|^N  \weak \nu$ weakly-$\star$ in $\M$;
\item
$\dd\alpha^n \to \dd \bara$ strongly in $W^{-1,N'}$;
\item
$\dd\beta^n \to \dd \barb$ strongly in $L^N$.
\end{enumerate}
Then we have the following convergence (modulo subsequences) in the sense of  distributions:
\begin{equation}
\alpha^n \wedge \beta^n \longrightarrow  \bara \wedge \barb + \sum_{k=1}^\infty \dd \left(v^k\delta_{x^k}\right) \qquad \text{in } \mathcal{D}'%\left(X, \bigwedge^{\ell_1+\ell_2-1}T^*X\right),
\end{equation}
for some sequences of points $\left\{x^k\right\} \subset X$ and $(\ell_1+\ell_2-1)$-covectors $\left\{v^k\right\}$. Moreover,
\begin{equation*}
\left|v^k\right| \lesssim_{N,X}
\left[\mu\left(\left\{x^k\right\}\right)\right]\left[\nu\left(\left\{x^k\right\}\right)\right]^{\frac{1}{N}}.
\end{equation*}
\end{theorem}

In this work, any theorem concerning $L^1$ will always be formulated for $\M$, with the $L^1$-norm of integrable functions replaced by the total variation of Radon measures.

\begin{remark}
For the endpoint critical wedge product Theorem~\ref{thm: endpoint}, one may suspect that Condition~(6) appears too strong. However, even replacing this by ``$\dd\beta^n \weak \dd \barb$ weakly in $L^N$'' would lead to counterexamples. See \cite[Example~3.4]{bcm}.
\end{remark}

Our proof of Theorem~\ref{thm: endpoint}  follows the strategies in \cite[\S 3]{bcm}, for which more delicate analysis for endpoint Sobolev embeddings are needed. Key tools include the endpoint estimates for elliptic  systems \emph{\`{a} la} Bourgain--Brezis \cite{bb04} and Brezis--Van Schaftingen \cite{bv}.

The two endpoint critical div-curl lemmas  by Briane--Casado-D\'{i}az--Murat \cite[Theorems~3.1 and 4.1]{bcm}, which were formulated and proved in rather different ways in \cite{bcm}, 
are generalised to the wedge product Theorem~\ref{thm: endpoint} in a unified manner. This is because our formulation, as per Robbin--Rogers--Temple \cite{key1}, circumvents  the asymmetry between divergence and curl by way of exploiting of the  instead (super)symmetric  wedge product.

%The following remark is immediate yet important for applications:
\begin{remark}
All the results in this paper remain valid if we drop the compactness assumption of $(X,g)$ and, meanwhile, change all the relevant function spaces $\mathcal{X}$ into the local spaces $\mathcal{X}_{\rm loc}$.
\end{remark}

In the final section \S\ref{sec: isom imm} of this paper, we apply the theoretical results established in previous sections to study the weak continuity of the compatibility equations for curvatures --- named after Gauss, Codazzi(--Mainardi), and Ricci --- of isometric immersions of Riemannian manifolds with extrinsic geometry of $L^p$-regularity.

More precisely, consider a family of (approximate, resp.) isometric immersions of a fixed Riemannian surface $(\Sigma,g)$ with $W^{2,p}$-regularity into $\R^3$. Then the associated second fundamental forms $\{\two^\e\}$, which are $2$-tensorfields on $\Sigma$ with $L^p$-regularity, are (approximate, resp.) weak solutions to the Gauss--Codazzi equations. One hopes to find the smallest possible index $p$ which ensures that any $L^p$-weak limit of $\{\two^\e\}$ remains a weak solution to the Gauss--Codazzi equations. The weak continuity property has been used to prove the existence of isometric immersions with weak regularity of certain negatively curved surfaces. See S. Mardare \cite{mar}, Chen--Slemrod--Wang \cite{csw0}, and the subsequent works \cite{csw, chw1, chw2, liarma, cg}. The aforementioned investigations have potential  applications to nonlinear elasticity. See \cite{elas1, elas2, elas3} and the many references therein.

We apply the generalised wedge product Theorem~\ref{thm: wedge} to establish the following:
\begin{theorem}\label{thm: isom imm, 2D}
Let $\left(\Sigma,g\right)$ be an immersed Riemannian surface in $\R^3$. Consider a family $\left\{\two^\e\right\} \subset L^p_\loc\left(\Sigma; \left[T\Sigma \otimes T\Sigma\right]^*\right)$  of weak solutions to the Gauss--Codazzi equations which converges to $\overline{\two}$ in the weak-$L^p_\loc$-topology, where $p \in \left]\frac{4}{3},\infty\right[$. Suppose that the coexact parts (\emph{i.e.}, projection onto the $\dd^*$-image via Hodge decomposition) of $\left\{\two^\e\right\}$ are precompact in the $L^{p'}_\loc$-topology. Then $\overline{\two}$ is still a weak solution to the Gauss--Codazzi equations. 
 \end{theorem}

Here we are able to go beyond the critical exponent $p_{\bf CS}=2$ as in \cite{csw0, csw} (which makes quadratic functions in $\two^\e$ well defined via the Cauchy--Schwarz inequality; see Remark~\ref{rem: low reg wedge}). The weak continuity of Gauss--Codazzi equations is established in the regime $p>4/3$, with compactness in higher regularity classes (\emph{i.e.} $L^{p'}_\loc$) assumed only for some components of $\two^\e$. For $p>2$, Theorem~\ref{thm: isom imm, 2D} agrees with the weak compactness results established in \cite{csw, cl1, cl2}.

We shall prove in \S\ref{sec: isom imm} a more general result (namely, Theorem~\ref{thm: isom imm}) than Theorem~\ref{thm: isom imm, 2D}, which applies to isometric immersions of arbitrary dimensions and codimensions.

%Two technical lemmas pertaining to concentration measures and endpoint elliptic regularity, which are  variants of well-known theorems   on Euclidean domains, are collected in Appendices.

To conclude the introduction, we comment  that the entire programme of the present paper is expected to carry over to more general metric measure spaces, especially RCD spaces. In addition, by combining the ideas in this paper and  Conti--Dolzmann--M\"{u}ller \cite{cdm+} (on an extended div-curl lemma, with relaxed first-order differential constraints and additional equi-integrability condition in quadratic combinations), we may further generalise the wedge product theorem of Robbin--Rogers--Temple \cite{key1}. The above topics are left for future investigation.

\section{Proof of the wedge product Theorem~\ref{thm: wedge}}\label{sec: proof of wedge}

This section is devoted to the proof of the generalised wedge product Theorem~\ref{thm: wedge}. 

For this purpose, we combine (with nontrivial modifications) ideas from Robbin--Rogers--Temple \cite{key1}  and Briane--Casado-D\'{i}az--Murat \cite{bcm}. We first exploit a nice ``substructure'' of $\alpha^n$ and $\beta^n$  based on the Hodge decomposition, which  explains why $\alpha^n\wedge\beta^n$ is well defined in the sense of distributions. The non-critical case $\frac{1}{p}+\frac{1}{q}<1+\frac{1}{N}$ follows from a direct  Sobolev embedding argument, while the critical case encompasses additional difficulties arising from  Sobolev embeddings with critical exponents, which will be treated via P.-L. Lions' theory of concentrated compactness \cite{lions}.

\subsection{Hodge-type decomposition of differential forms}

We say that a pair of sequences $\left(\left\{x^n\right\},\left\{y^n\right\}\right)$ is a \emph{``weak-strong'' duality pairing} if for some $r \in ]1,\infty[$, one of $\{x^n\}$ or $\{y^n\}$ converges weakly in $L^r$ and the other converges strongly in $L^{r'}$. Hence, for a quadratic form $Q$, we have that $\{Q(x^n,y^n)\}$ converges weakly in $L^1$. Similarly we shall speak of ``weak-weak'' or ``strong-strong'' pairings, which are by now self-explanatory.

Proposition~\ref{prop: decomposition of diff forms in wedge prod} below is essentially \cite[ Lemma~4.1]{key1}, generalised from Euclidean domains to compact manifolds and with  $(p,q)$ relaxed from H\"{o}lder conjugate exponents to the critical indices satisfying \eqref{p,q condition}. Our proof is slightly different from the corresponding arguments in \cite{key1, bcm}.

\begin{proposition}\label{prop: decomposition of diff forms in wedge prod}
Let $\left(X^N,g\right)$, $p$, $q$, $\left\{\alpha^n\right\}$, and $\left\{\beta^n\right\}$ be as in Theorem~\ref{thm: wedge}. One can find  differential forms $\gamma_n \in W^{1,p}\left(X, \bigwedge^{\ell_1 - 1} T^*X \right)$, $\xi^n \in L^{q'}\left(X, \bigwedge^{\ell_1} T^*X \right)$, $\zeta_n \in W^{1,q}\left(X, \bigwedge^{\ell_2-1} T^*X \right)$, and $\eta^n \in L^{p'}\left(X, \bigwedge^{\ell_2} T^*X \right)$ for each $n = 1,2,3,\ldots$ such that 
\begin{enumerate}
\item
the following decompositions hold \emph{a.e.}:
\begin{equation*}
\begin{cases}
\alpha^n = \dd \gamma_n + \xi^n,\\
\beta^n = \dd \zeta_n + \eta^n;
\end{cases}
\end{equation*}
\item
the following co-exactness conditions holds in the sense of distributions: 
\begin{equation*}
\dd^* \xi^n = 0 \quad \text{and} \quad \dd^* \eta^n = 0\qquad \text{ in } \mathcal{D}';
\end{equation*}
\item
the following convergence results hold (modulo subsequences) in  indicated topologies:
\begin{eqnarray*}
&& \gamma_n \weak \gamma \qquad \text{ weakly in } W^{1,p}\left(X, \bigwedge^{\ell_1 - 1} T^*X \right),\\
&& \xi^n \longrightarrow \xi\qquad \text{ strongly in } L^{q'}\left(X, \bigwedge^{\ell_1} T^*X \right),\\
&&\zeta_n \weak \zeta \qquad \text{ weakly  in } W^{1,q}\left(X, \bigwedge^{\ell_2-1} T^*X \right),\\
&& \eta^n \longrightarrow \eta  \qquad \text{ strongly in } L^{p'}\left(X, \bigwedge^{\ell_2} T^*X \right).
\end{eqnarray*}
\end{enumerate}
\end{proposition}

Here and hereafter the following notational convention is adopted in accordance with \cite{bcm}.

\begin{notation}
For the terms obtained from decomposing given weakly convergent sequences $\left\{\alpha^n\right\}$ and $\left\{\beta^n\right\}$, those with superscripts (subscripts resp.) converge strongly (weakly resp.) with respect to natural topologies given in the context.
\end{notation}

\begin{proof}[Proof of Proposition~\ref{prop: decomposition of diff forms in wedge prod}]

 We first observe that condition~\eqref{p,q condition} on the range of $(p,q)$ implies
 \begin{align*}
 1 < p \leq q'\qquad \text{ and } \qquad 1 < q \leq p'.
\end{align*}
The results for $\alpha^n$ and $\beta^n$ are completely parallel, so we argue for $\alpha^n$ only.

Indeed, by classical Hodge decomposition we have
\begin{align}\label{hodge for alpha n}
\alpha^n = \dd \gamma_n + \dd^* k^n + h^n
\end{align}
with differential $(\ell_1 + 1)$-forms $\left\{k^n\right\}$ of $W^{1,p}$-regularity, and harmonic $\ell_1$-forms $\left\{h^n\right\}$ of $L^p$-regularity. Thanks to \cite[p.86, Theorem~2.4.7]{sch},  here one may further select $\gamma_n$ and $k^n$ to be ``minimal'' in the sense that
\begin{align*}
\text{ $\dd^*\gamma_n = 0$ and $\dd k^n=0$ weakly. }
\end{align*}
Then 
\begin{align*}
\dd \alpha^n = \dd\dd^* k^n = \Delta k^n,
\end{align*}
where $\Delta = \dd^*\dd + \dd\dd^*$ is the Laplace--Beltrami operator on $(X,g)$. By the ellipticity of $\Delta$ and the assumption that $\left\{\dd\alpha^n\right\}$ converges strongly in $W^{-1,q'}$, we deduce that 
\begin{equation}\label{kn conv}
\text{$\{k^n\}$ converges strongly in $W^{1,q'}$.}
\end{equation}
 Notice that although $k^n$ may fail to be uniquely soluble from $\dd \alpha^n = \Delta k^n$, it is determined up to addition by harmonic forms, which is immaterial to \eqref{hodge for alpha n} as $\dd^*k^n$ is uniquely determined. 

On the other hand, the projection $\pi$ from $L^p\left(X, \bigwedge^{\ell_1} T^*X \right)$ onto the subspace $\mathcal{H}$ of harmonic $L^p$-forms [$\pi(\alpha^n)=h^n$] is compact, for the range of $\pi$ is finite-dimensional. This follows from the $\CC^\infty$-regularity of $L^p$-harmonic forms and the fact that  $\mathcal{H}$ here coincides with the $\ell_1^{\text{th}}$-cohomology group. In this regard, one may equip $\mathcal{H}$ with the $\CC^\infty$-(Fr\'{e}chet) topology to maintain the compactness of $\pi$. See \cite[\S 2.6]{sch} for rudiments of de Rham theory. 
 In particular, 
\begin{equation}\label{hn conv}
\text{$\left\{h^n\right\}$ converges strongly in $L^{q'}$.}
\end{equation}

Therefore, setting
\begin{align*}
\xi^n := \dd^*k^n+h^n,
\end{align*}
we have its strong convergence in $L^{q'}$ from \eqref{kn conv} and \eqref{hn conv}. Also, since $\dd^* \circ \dd^*=0$ and $h^n$ is harmonic, we find that $\dd^*\xi^n=0$ in the sense of distributions.

Finally, taking $\dd^*$ to \eqref{hodge for alpha n} and recalling $\dd^*\xi^n=0$ as well as the choice of $\gamma_n$, we have
\begin{align*}
\dd^*\alpha^n = \dd^*\dd\gamma_n=\Delta \gamma_n.
\end{align*}
As $\left\{\alpha^n\right\}$ is weakly convergent in $L^p$ by assumption, using the ellipticity  of $\Delta$, we deduce that $\left\{\gamma_n\right\}$ is weakly convergent in $W^{1,p}$. Again, $\gamma_n$ is only determined up to harmonic forms, but then $\dd \gamma_n$ is uniquely determined; so, in view of \eqref{hodge for alpha n}, without loss of generality we may fix the cohomology class of $\gamma_n$ once and for all.  \end{proof}

\subsection{Definition of the weak-weak pairing}\label{subsec: weak weak pairing}

Proposition~\ref{prop: decomposition of diff forms in wedge prod} allows us to define $\alpha^n \wedge \beta^n$ in the sense of distributions, utilising the algebraic structure of wedge product (bilinearity and superdistributivity under exterior differential). As commented in Remark~\ref{rem: low reg wedge}, without probing into the ``substructures'' of $\alpha^n$ and $\beta^n$, one has apparently no sufficient regularity  to define their wedge products as   distributions.

\subsubsection{Weak-weak paring is well defined in the sense of distributions} 
First, due to the  bilinearity of wedge product we (formally should) have
\begin{align}\label{formal}
\alpha^n \wedge \beta^n =  {\dd \gamma_n \wedge \dd\zeta_n} + \xi^n \wedge \dd\zeta_n + \dd\gamma_n \wedge \eta^n + \xi^n \wedge \eta^n.
\end{align}
The last three terms on the right-hand side of \eqref{formal}, being weak-strong or strong-strong duality pairings, all converge weakly in $L^1\left(X,\bigwedge^{\ell_1 + \ell_2}T^*X\right)$.

The most singular term is the weak-weak pairing $\dd \gamma_n \wedge \dd\zeta_n$. We ``differentiate by parts'' to find that it is well defined in the sense of distributions. Indeed, by $\dd\circ\dd = 0$ in $\mathcal{D}'(X)$ and the super-distributivity of $\dd$ over wedge product, we have
\begin{align}\label{bad}
\dd \gamma_n \wedge \dd\zeta_n = \dd\left(\gamma_n \wedge \dd \zeta_n\right),
\end{align}
provided that $\gamma_n$ and $\zeta_n$ are of class  $\CC^1$. But $\{\dd\zeta_n\}$ is contained in $L^q$, while $\{\gamma_n\}$ lies in $L^{p^*}$ by virtue of Sobolev embedding (as usual, $p^*=\frac{Np}{N-p}$ denotes the Sobolev conjugate exponent). By assumption we have that $\left[q^{-1} + (p^*)^{-1}\right]^{-1} = \left[p^{-1} +q^{-1}-N^{-1}\right]^{-1} \geq 1$, so $\gamma_n \wedge \dd \zeta_n$ is a well defined $L^1$-function \emph{for each fixed $n$}. Hence, one may now define $\dd \gamma_n \wedge \dd\zeta_n$ via \eqref{bad} as a differential $(\ell_1+\ell_2)$-form-valued) distribution.

Of course, the argument above does not ensure that $\gamma_n \wedge \dd \zeta_n$ (hence $\dd \gamma_n \wedge \dd\zeta_n$) tend to any limit as $n\to\infty$ in the sense of distributions, as the weak $L^1$-topology is not compactly embedded in $\mathcal{D}'$.

\subsubsection{ Independence of representation}
 We check that the definition of $\alpha^n \wedge \beta^n$ given by \eqref{formal} $\&$ \eqref{bad} is independent of the decomposition of $\alpha^n$ and $\beta^n$ as in  Proposition~\ref{prop: decomposition of diff forms in wedge prod}. 

\begin{lemma}
Let $\alpha \in L^{p}\left(X, \bigwedge^{\ell_1}T^*X\right)$ and $\beta \in L^{q}\left(X, \bigwedge^{\ell_2}T^*X\right)$ be differential forms on a closed Riemannian manifold $\left(X^N,g\right)$, where $p,q$ satisfy \eqref{p,q condition}. Assume the decompositions
\begin{equation*}
\begin{cases}
\alpha = \dd\gamma + \xi = \dd \gamma' + \xi',\\
\beta = \dd\zeta + \eta = \dd\zeta' + \eta',
\end{cases}
\end{equation*}
where $\gamma, \gamma' \in W^{1,p}\left(X, \bigwedge^{\ell_1 - 1} T^*X \right)$, $\xi,\xi' \in L^{q'}\left(X, \bigwedge^{\ell_1} T^*X \right)$, $\zeta, \zeta' \in W^{1,q}\left(X, \bigwedge^{\ell_2-1} T^*X \right)$, and $\eta, \eta' \in L^{p'}\left(X, \bigwedge^{\ell_2} T^*X \right)$. Also, assume that $\dd^*\xi=\dd^*\xi'=0$ and $\dd^*\eta=\dd^*\eta'=0$ in the sense of distributions. Then it holds, in the sense of distributions too, that
\begin{align*}
\dd\left(\gamma\wedge\dd\zeta\right) + \xi \wedge \dd\zeta + \dd\gamma \wedge \eta + \xi \wedge \eta = \dd\left(\gamma'\wedge\dd\zeta'\right) + \xi' \wedge \dd\zeta' + \dd\gamma' \wedge \eta' + \xi' \wedge \eta'.
\end{align*}
\end{lemma}

\begin{proof}
Our argument is an adaptation and generalisation of the arguments in \cite[p.480]{bcm}.

First let us check that 
   \begin{align}\label{a'}
   \dd\left(\gamma\wedge\dd\zeta\right) + \xi \wedge \dd\zeta + \dd\gamma \wedge \eta + \xi \wedge \eta = \dd\left(\gamma'\wedge\dd\zeta\right) + \xi' \wedge \dd\zeta + \dd\gamma' \wedge \eta + \xi' \wedge \eta.
   \end{align}
Indeed, the difference between the members on the left and right hand sides equals
\begin{align*}
D := \dd\left[\left(\gamma-\gamma'\right)\wedge\dd\zeta\right] + \dd\left(\gamma-\gamma'\right) \wedge \eta + \left(\xi-\xi'\right)\wedge \beta.
\end{align*}
Noting that $$\dd\left(\gamma-\gamma'\right) = \xi'-\xi$$ due to the decompositions of $\alpha$, we combine the last two terms to deduce that
\begin{align*}
D = \dd\left[\left(\gamma-\gamma'\right)\wedge\dd\zeta\right] +  \left(\xi-\xi'\right)\wedge\dd\zeta.
\end{align*}
Now, for the first term on the right-hand side we have $\left(\gamma-\gamma'\right)\wedge (\dd\dd\zeta)=0$ in the sense of distributions, thanks to  $\gamma-\gamma' \in W^{1,p}$ and $\zeta \in W^{1,q}$ where $p^{-1} + q^{-1} \leq 1 + N^{-1}$. Hence, 
\begin{align*}
D = \dd\left(\gamma-\gamma'\right)\wedge\dd\zeta  +  \left(\xi-\xi'\right)\wedge\dd\zeta.
\end{align*}
Finally, using the fact that $\dd\left(\gamma-\gamma'\right)=\xi'-\xi$, we obtain
\begin{align*}
D=0.
\end{align*}

%In the above, one may be skeptical about the pairing $\left(\xi-\xi'\right)\wedge\dd\zeta$, since $\xi-\xi' \in L^p$ and $\dd\zeta \in L^q$, while $\frac{1}{p}+\frac{1}{q}$ is only known to be greater than or equal to $1$. However, since $\dd\left(\gamma-\gamma'\right)=\xi'-\xi$, one has $\dd\left(\xi'-\xi\right)=0$. On the other hand, by assumption we have $\dd^*\left(\xi'-\xi\right)=0$. So $\xi'-\xi$ is harmonic on $X$, and hence is $\CC^\infty$ by elliptic regularity.

Next we check that 
\begin{align}\label{a''}
\dd\left(\gamma'\wedge\dd\zeta\right) + \xi' \wedge \dd\zeta + \dd\gamma' \wedge \eta + \xi' \wedge \eta = \dd\left(\gamma'\wedge\dd\zeta'\right) + \xi' \wedge \dd\zeta' + \dd\gamma' \wedge \eta' + \xi' \wedge \eta'.
\end{align}
Indeed, the difference between left- and right-hand sides equals
\begin{align*}
&\dd\left(\gamma'\wedge\dd\zeta\right) - \dd \left(\gamma'\wedge\dd\zeta'\right) + \dd\gamma' \wedge\left(\eta-\eta'\right)\\
&= \dd\left[\gamma'\wedge\dd\left(\zeta-\zeta'\right)\right] + \dd\gamma' \wedge\left(\eta-\eta'\right)\\
&= \dd\gamma' \wedge\left[\dd\left(\zeta-\zeta'\right) + \left(\eta-\eta'\right)\right]\\
&=0.
\end{align*}

The lemma now follows from \eqref{a'} $\&$ \eqref{a''}, both understood in the distributional sense.  
\end{proof}

\subsection{Proof for the subcritical case: $1<p,q<\infty$ and $1 \leq \frac{1}{p}+\frac{1}{q} < 1+\frac{1}{N}$}\label{subsec: subcrit}

In this case we have $q'<p^*$, the Sobolev conjugate of $p$. Thus the embedding $W^{1,p}\left(X, \bigwedge^{\ell_1-1} T^*X \right)\emb\emb L^{q'}\left(X, \bigwedge^{\ell_1-1} T^*X \right)$ is compact. So, $\gamma_n \wedge \dd\zeta_n$ in \eqref{bad} becomes a \emph{strong-strong} pairing of $L^q$- and $L^{q'}$-differential forms. The wedge product converges strongly in $L^1$ by H\"{o}lder's inequality, hence $\dd \gamma_n \wedge \dd\zeta_n$ (appropriately interpreted as in \S\ref{subsec: weak weak pairing}) converges strongly in $W^{-1,1}$. This together with the weak-$L^1$-convergence of the other three terms on the right-hand side of \eqref{formal} implies that $\alpha^n \wedge \beta^n$ converges in the sense of distributions. This proves \eqref{thm: wedge} in the ``sub-critical'' case.

\subsection{Proof for the (non-endpoint) critical case: $1<p,q<\infty$ and $\frac{1}{p}+\frac{1}{q} = 1+\frac{1}{N}$}\label{subsec, where conc comp is used}

As in the above proof for the subcritical case, the only term needs to be treated is $\alpha^n \wedge \beta^n$ in \eqref{formal}. In the ``non-endpoint'' critical case (\emph{i.e.}, $p \neq 1 \neq q$), we have 
\begin{equation}\label{indices}
1 < p \leq q',\quad 1 < q \leq p',\quad 1<p<N,\quad 1<q<N,\quad p'=q^*,\quad \text{and}\quad q'=p^*. 
\end{equation}

Due to the assumptions (see Theorem~\ref{thm: wedge}, (3)) and Proposition~\ref{prop: decomposition of diff forms in wedge prod}, it holds that
\begin{equation}\label{d gamma convergence}
\left|\dd\left(\gamma_n-\gamma\right)\right|^p\dvg \weak \mu \qquad \text{weakly-$\star$ in $\M(X)$}. 
\end{equation}
Here we used that $\xi^n\to\xi$ strongly in $L^{q'}$, hence strongly in $L^p$, as $X$ is compact and $1<p \leq q'$ by \eqref{indices}. Meanwhile, we may choose $\gamma_n$ and $\gamma$ to be coexact; see \cite[p.86, Theorem~2.4.7]{sch} and the arguments in the proof of Proposition~\ref{prop: decomposition of diff forms in wedge prod}.   This observation, together with \eqref{d gamma convergence} and the classical Gaffney's inequality --- \emph{i.e.}, the $W^{1,q}$-norm of a differential form $\omega$ can be bounded by the $L^q$-norms of $\dd\omega$, $\dd^*\omega$, and $\omega$; see \cite{gaff1, gaff2, gaff3, gaff4} ---  shows that $\left\{\gamma_n\right\}$ is bounded in $W^{1,p}\left(X, \bigwedge^{\ell_1-1} T^*X \right)$. 

We are in the situation of the critical, noncompact Sobolev embedding $W^{1,p} \emb L^{p^*}=L^{q'}$; see \eqref{indices}. By the second concentration compactness lemma \emph{\`{a} la} P.-L. Lions \cite[p.158, Lemma~I.1]{lions}, we have countably many points $\{x^j\} \subset X$ and nonnegative constants $\{c_j\}$ such that
\begin{align}\label{conc comp}
&\left|\gamma
_n - \gamma\right|^{q'}\dvg \weak \lambda' := \sum_{j=1}^\infty c_j \delta_{x^j} \qquad \text{weakly-$\star$ in $\M(X)$},\nonumber\\
&\qquad\qquad \text{ with } \left(c_j\right)^{{p}\slash{q'}} \lesssim_{p,N,X} \mu\left(\left\{x^j\right\}\right) \quad\text{ for each $j$}. 
\end{align}

On the other hand, by assumption (Theorem~\ref{thm: wedge}, (4)) and Proposition~\ref{prop: decomposition of diff forms in wedge prod}, we have 
\begin{equation*}
\left|\dd\left(\zeta_n - \zeta\right)\right|^{q}\dvg\weak \nu \qquad \text{weakly-$\star$ in $\M(X)$}. 
\end{equation*}

An application of Lemma~\ref{lem: appendix} in the Appendix yields a measure $$\eth \in \M\left(X; \bigwedge^{\ell_1+\ell_2 - 1} T^*X  \right)$$
such that
\begin{equation}\label{defect meas}
\gamma_n\wedge\dd\zeta_n \weak \gamma \wedge \dd\zeta + \eth \qquad \text{weakly-$\star$ in $\M(X)$},
\end{equation}
and that, for each Borel set $E \subset X$, the total variation of $\eth$ is estimated by 
\begin{equation*}
\left|\eth\right|(E) \leq \left[\lambda'(E)\right]^{\frac{1}{q'}} \left[\nu(E)\right]^{\frac{1}{q}}. 
\end{equation*}
By the structure of $\lambda'$ characterised in \eqref{conc comp} via concentration compactness, we have
\begin{align*}
\left|\eth\right|(E) &\leq  \left[\nu(E)\right]^{\frac{1}{q}} \left\{ \sum_{\{j:\,x^j \in E\}} c_j \right\}^{\frac{1}{q'}}\nonumber\\
&\lesssim_{p,N,X}\,\sum_{\{j:\,x^j \in E\}} \left[\mu\left(\left\{x^j\right\}\right)\right]^{\frac{1}{p}}\left[\nu\left(\left\{x^j\right\}\right)\right]^{\frac{1}{q}}. 
\end{align*}
The defect measure $\eth$ is thus atomic:
\begin{align}\label{defect meas char}
\eth = \sum_{j=1}^\infty v^j \delta_{x^j} \quad \text{with}\quad
\left|v^j\right| \lesssim_{p,N,X} \left[\mu\left(\left\{x^j\right\}\right)\right]^{\frac{1}{p}}\left[\nu\left(\left\{x^j\right\}\right)\right]^{\frac{1}{q}}\quad \text{ for each $j$}.
\end{align}

Note also that \eqref{defect meas char} together with the assumptions in Theorem~\ref{thm: wedge} (1)--(4) (as well as the lower semicontinuity of Radon measures) implies that $$\sum_{j=1}^\infty \left|v^j\right|<\infty,$$ so $\eth$ is a well defined distribution.

We may now conclude \eqref{conclusion of wedge prod thm} and the ensuing estimates from \eqref{defect meas char} and \eqref{defect meas}. The proof of Theorem~\ref{thm: wedge} is now complete.

\begin{remark}
In the above proof, all the lengths of vectorfields and differential forms on $X$ are taken with respect to the Riemannian metric. This is in agreement with \cite[Remark I.5]{lions}, for we are applying concentration compactness to $(X,g)$ as a metric measure space. This observation motivates the concluding paragraph in the Introduction \S\ref{sec: intro}. 
\end{remark}

\section{Consequences of the Wedge Product Theorem~\ref{thm: wedge}}\label{sec: consequences of wedge prod thm}

\subsection{A multilinear wedge product theorem} By an inductive argument, we obtain the following generalisation of the multiple wedge product theorem in Robbin--Rogers--Temple \cite[Theorem~1.1]{key1} (see Lemma~\ref{lemma: RRT} above).

For its formulation  one more notation is needed: for an $L^p$-differential form $\omega \in L^p\left(X,\bigwedge^\bullet T^*X\right)$ with $1<p<\infty$, Hodge decomposition (see, \emph{e.g.}, \cite{sch}) yields that
\begin{align*}
\omega = \dd \omega_1 + \dd^*\omega_2 + \omega_h,
\end{align*}
where $\omega_h$ is harmonic and $\dd\omega_1$, $\dd^*\omega_2$, and $\omega_h$ are $L^p$-forms. Moreover, one has the uniqueness result: if it also holds that $\omega = \dd\omega_1' + \dd^*\omega_2' + \omega_h'$ with $\dd\omega_1'$, $\dd^*\omega_2'$, $\omega_h' \in L^p$ and $\omega_h'$  harmonic, then $\dd\omega_1=\dd\omega_1'$, $\dd^*\omega_2 = \dd^*\omega_2'$, and $\omega_h=\omega_h'$. So there is a well-defined norm-1 projection 
\begin{align}\label{proj}
\proj:  L^p\left(X,\bigwedge^\bullet T^*X\right) \longrightarrow  L^p\left(X,\bigwedge^\bullet T^*X\right),\qquad \proj(\omega) := \dd\omega_1.
\end{align}
That is, $\proj$ is the projection of $L^p$-differential forms onto the exact part in Hodge decomposition.

%In other words, by Hodge decomposition  each $L^p$-form can be uniquely decomposed into the sum of $\dd$-, $\dd^*$-, and harmonic parts. Then $\proj$ is the projection onto the $\dd$-part. 
 
Our multiple wedge product theorem is stated below. We have an additional ``no-loss in norm'' condition, which we currently do not know whether it can be dropped.

\begin{theorem}\label{thm: multiple wedge prod}
Let $\left(X^N,g\right)$ be an $N$-dimensional closed Riemannian manifold; $N \geq 2$. Let $1<p_1, \ldots, p_L<\infty$ be such that $$1\leq \sum_{i=1}^L \frac{1}{p_i} \leq 1+\frac{1}{N},$$ where $L$ is a natural number, and assume that for each $i \in \{1,\ldots,L\}$, one has $$q_i := \frac{1}{\sum_{j \in \{1,\ldots,L\}\setminus\{i\}} \frac{1}{p_j}}>1.$$

Consider $L$ sequences of differential forms $\left\{ \alpha_{1}^n \right\}, \ldots, \left\{ \alpha_{L}^n \right\}$, where $ \alpha_{i}^n \in L^{p_i}\left( X,\bigwedge^{\ell_i} T^*X \right)$ for each $n$ and $0 \leq \ell_i \leq n$, that satisfy the following conditions:
\begin{enumerate}
\item
$\alpha_i^n \weak \bara_i$ weakly in $L^{p_i}$ for each $i \in \{1,2,\ldots,L\}$;
\item
$\left|\alpha^n_i - \bara_i\right|^{p_i}\dvg \weak \mu_i$ weakly-$\star$ in $\M$ for each $i \in \{1,2,\ldots,L\}$;
\item
$\dd \alpha_i^n \to \dd \bara_i$ strongly in $W^{-1,q_i'}$ for each $i \in \{1,2,\ldots,L\}$; and
\item (``no-loss in norm'')
there exists $k \in \{1,2,\ldots,L\}$ such that 
\begin{align}\label{no loss}
&\left\|\alpha^n_1 \wedge \ldots \wedge \widehat{\alpha^n_k} \wedge \ldots \wedge \alpha^n_L - \proj\alpha^n_1 \wedge \ldots \wedge \widehat{\proj\alpha^n_k} \wedge \ldots \wedge \proj\alpha^n_L \right\|_{L^{p_k'}}\nonumber\\
&\quad \longrightarrow \left\|\bara_1 \wedge \ldots \wedge \widehat{\bara_k} \wedge \ldots \wedge \bara_L - \proj\bara_1 \wedge \ldots \wedge \widehat{\proj\bara_k} \wedge \ldots \wedge \proj\bara_L \right\|_{L^{p_k'}} \quad\text{as } n \to \infty,
\end{align}
wherein the projection operator $\proj$ has been defined in \eqref{proj}.
\end{enumerate}

Then the convergence result below holds after passing to subsequences:
\begin{align*}
\alpha^n_1 \wedge \ldots \wedge \alpha^n_L \longrightarrow \bara_1 \wedge \ldots\wedge \bara_L + \sum_{k=1}^\infty\dd\left(v^k \delta_{x^k}\right) \qquad \text{in } \mathcal{D}'%\left(X,\bigwedge^{\sum_{j=1}^L \ell_j} T^*X\right)
\end{align*}
for some sequences of points $\left\{x^k\right\} \subset X$ and $\left(\sum_{j=1}^L \ell_j - 1\right)$-covectors $\left\{v^k\right\}$. In addition, the weights $\left\{v^k\right\}$ have their moduli bounded as follows:
\begin{align*}
\left|v^k\right| \lesssim_{p_1,\ldots,p_L,N,X} \left\{ \prod_{j=1}^L \left[\mu^j\left(\left\{x^k\right\}\right)\right]^{\frac{1}{p_j}}\right\}.
\end{align*} 
\end{theorem}

\begin{notation}
Here and hereafter, %the moduli $|\bullet|$ of differential forms are taken with respect to the Riemannian metric $g$, and 
 circumflex in expressions of the form $\mathfrak{b_1} \wedge \ldots \wedge \widehat{\mathfrak{b}_k} \wedge \ldots \wedge \mathfrak{b}_L$ (and the like)  designates the omission of the $k^{\text{th}}$ entry $\mathfrak{b}_k$. 
\end{notation}

\begin{proof}[Proof of Theorem~\ref{thm: multiple wedge prod}]
First of all, by supercommutativity of the wedge product, there is no loss of generality to prove for only one $k$ in the ``no-loss in norm'' condition~(4). Throughout the proof we shall fix $k=L$.   The thesis follows immediately from Theorem~\ref{thm: wedge} by identifying $\left(\alpha^n_1 \wedge \ldots \wedge \alpha^n_{L-1}, \alpha^n_L, \bara_1 \wedge \ldots\wedge \bara_{L-1},\bara_L\right)$ with $\left( \beta^n,\alpha^n, \barb,\bara\right)$ and $(q_L, p_L)$ with $(q,p)$, \emph{provided that the hypotheses in Theorem~\ref{thm: wedge} are verified}.  The estimate on $v^k$ are proved by a simple induction argument, of which the case $L=2$ is covered by Theorem~\ref{thm: wedge}.

 The verification for the hypotheses in Theorem~\ref{thm: wedge} is carried out below in  two steps.

\smallskip
\noindent
{\bf Step 1.} First, observe that $\left\{\alpha^n_1 \wedge \ldots \wedge \alpha^n_{L-1}\right\}$ is bounded in $L^{q_L}\left(X,\bigwedge^{\sum_{j=1}^{L-1} \ell_j}T^*X\right)$, as each $\left\{\alpha^n_i\right\}$ is bounded in $L^{p_i}$ by assumption. We are in the subcritical case in the sense of \S\ref{subsec: subcrit} (for only $(L-1)$ forms are taken into account here, so that $1\leq \sum_{i=1}^{L-1} \frac{1}{p_i} < 1+\frac{1}{N}$). Thus, a direct induction and Theorem~\ref{thm: wedge} yield that $\alpha^n_1 \wedge \ldots \wedge \alpha^n_{L-1} \to \bara_1 \wedge \ldots \wedge \bara_{L-1}$ in the distributional sense. This together with the $L^{q_L}$-boundedness established above implies that the convergence takes place in the weak $L^{q_L}$-topology. 

As a result, by a standard compactness argument, there exists a positive finite Radon measure $\mu$ on $X$ such that
\begin{align*}
\left|\alpha^n_1 \wedge \ldots \wedge \alpha^n_{L-1} - \bara_1 \wedge \ldots \wedge \bara_{L-1}\right|^{q_L}\dvg \weak \mu \qquad \text{ weakly-$\star$ in $\M(X)$}.
\end{align*}

\smallskip
\noindent
{\bf Step 2.} We also need to justify that 
\begin{align*}
\text{$\left\{\dd\left(\alpha^n_1 \wedge \ldots \wedge \alpha^n_{L-1}\right)\right\}$ converges strongly in $W^{-1, p_L'}$.}
\end{align*}
This essentially relies on the cancellation phenomenon exploited in \S\ref{subsec: weak weak pairing}, and the proof will be presented  in the four substeps~2A--2D below.

\noindent
{\bf Step~2A.}  By Proposition~\ref{prop: decomposition of diff forms in wedge prod}, one can decompose each $\alpha^n_j$ for $n \in \mathbb{N}$ and $j \in \{1,\ldots, L-1\}$ into
\begin{align*}
\alpha^n_j = \dd\gamma_{n,j} + \xi^n_j,
\end{align*}
where
\begin{equation*}
\begin{cases}
\gamma_{n,j} \weak \gamma_j\qquad \text{ weakly in } W^{1,p_j},\\
\xi^n_j \longrightarrow \xi_j\qquad \text{ strongly in } L^{q_j'}. 
\end{cases} 
\end{equation*}
We need to show that 
\begin{align}\label{strong conv, to prove}
&\dd\left\{\alpha^n_1 \wedge \ldots \wedge \alpha^n_{{L-1}}\right\} = \dd\left\{\left(\dd\gamma_{n,1} + \xi^n_1\right) \wedge \ldots \wedge  \left(\dd\gamma_{n,L-1} + \xi^n_{L-1}\right)\right\} \nonumber\\
&\quad \longrightarrow \dd\left\{\left(\dd\gamma_{1} + \xi_1\right) \wedge \ldots \wedge  \left(\dd\gamma_{L-1} + \xi_{L-1}\right)\right\} = \dd\left\{\bara_1 \wedge \ldots \wedge \bara_{L-1}\right\} \text{ strongly in } W^{-1,p_L'}.
\end{align}

\begin{remark}
Of course, \eqref{strong conv, to prove} would follow immediately if one could prove that
\begin{align}\label{false}
&\alpha^n_1 \wedge \ldots \wedge \alpha^n_{{L-1}} = \left(\dd\gamma_{n,1} + \xi^n_1\right) \wedge \ldots \wedge  \left(\dd\gamma_{n,L-1} + \xi^n_{L-1}\right)\nonumber\\
&\quad \longrightarrow \left(\dd\gamma_{1} + \xi_1\right) \wedge \ldots \wedge  \left(\dd\gamma_{L-1} + \xi_{L-1}\right) = \bara_1 \wedge \ldots \wedge \bara_{L-1} \qquad\text{ strongly in } L^{p_L'}.
\end{align} 
This, unfortunately, is impossible --- the term $\dd\gamma_{n,1} \wedge \ldots \wedge \dd\gamma_{n,L-1}$ is an $(L-1)$-multilinear combination of weakly convergent sequences in $L^{p_j}$ for $j \in \{1,\ldots, L-1\}$, which by H\"{o}lder's inequality is only bounded in $$L^{\left[\sum_{j=1}^{L-1} \frac{1}{p_j}\right]^{-1}} = L^{q_L}.$$ But as for \eqref{indices} we have $q_L \leq p_L'$, so the false claim~\eqref{false} is stronger than what we can actually obtain. 
\end{remark}

\noindent
{\bf Step~2B.}   Nevertheless, $\dd\gamma_{n,1} \wedge \ldots \wedge \dd\gamma_{n,L-1}$ is the \emph{only} troublesome term here --- we assert that
\begin{align}\label{assert}
&\text{all the other $\left(2^{L-1}-1\right)$ terms in $\left(\dd\gamma_{n,1} + \xi^n_1\right) \wedge \ldots \wedge  \left(\dd\gamma_{n,L-1} + \xi^n_{L-1}\right)$}\nonumber\\
&\text{converge strongly in $L^{p_L'}$.}
\end{align}

To verify the assertion, first we note that 
\begin{align*}
&\bigg\{\text{the sum of all such $\left(2^{L-1}-1\right)$ terms in $\left(\dd\gamma_{n,1} + \xi^n_1\right) \wedge \ldots \wedge  \left(\dd\gamma_{n,L-1} + \xi^n_{L-1}\right)$}\bigg\} \\
&\qquad \equiv  \left(\alpha^n_{1} \wedge \ldots \wedge \alpha^n_{L-1}\right) - \left(\proj\alpha^n_{1} \wedge \ldots \wedge \proj\alpha^n_{L-1}\right).
\end{align*}
This follows from the definition of $\proj$, since it holds for each $j \in \{1,\ldots,L-1\}$ that $$\proj \alpha^n_j = \dd\gamma_{n,j}.$$

In addition, if we can prove that $$\left\{\left(\alpha^n_{1} \wedge \ldots \wedge \alpha^n_{L-1}\right) - \left(\proj\alpha^n_{1} \wedge \ldots \wedge \proj\alpha^n_{L-1}\right)\right\}$$ converges \emph{weakly} in $L^{p_L'}$, then the ``no-loss in norm'' condition~\eqref{no loss} will promote it to \emph{strong} convergence in $L^{p_L'}$, by virtue of the Radon--Riesz theorem. See, \emph{e.g.}, \cite[p.17, Theorem~1.37]{afp}.

For this purpose, notice that any of the $\left(2^{L-1}-1\right)$ terms in consideration takes the form $\mathfrak{a}_{n,1} \wedge \ldots \wedge \mathfrak{a}_{n,L-1}$, where
\begin{itemize}
\item
each $\mathfrak{a}_{n,j}$ is either $\dd\gamma_{n,j}$ or $\xi^n_j$; and
\item
for at least one $j \in \{1,\ldots, L-1\}$, $\mathfrak{a}_{n,j} = \xi^n_j$.
\end{itemize}
It suffices to prove the convergence of those terms with $\mathfrak{a}_{n,j} = \xi^n_j$ for \emph{only one} $j \in \{1,\ldots, L-1\}$, because at each place $\mathfrak{a}_{n,j}$ the choice $\xi_j^n$ has higher regularity than $\dd\gamma_{n,j}$. In fact, $\left\{\xi_j^n\right\}$ is strongly convergent in $L^{q_j'}$ and $\left\{\dd\gamma_{n,j}\right\}$ is weakly convergent in $L^{p_j}$, with the indices $p_j \leq q_j'$ (again argued as for \eqref{indices}). Thus, as the underlying space $\left(X,\dvg\right)$ is compact, the more $\xi_j^n$ we have in $\mathfrak{a}_{n,j}$, the better convergence results one may obtain. 

Summarising the above discussions, it now suffices to show  that
\begin{align}\label{conv, true}
\left\{\dd \gamma_{n,1} \wedge \ldots \wedge \dd\gamma_{n, j-1} \wedge \xi^n_j \wedge \dd\gamma_{n,j+1} \wedge \ldots \wedge \dd \gamma_{n, L-1}\right\}\quad\text{converges weakly in } L^{p_L'}
\end{align}
for each $j \in \{1,\ldots, L-1\}$.

In the $(L-1)$-multilinear combination above, we have one strongly convergent term $\xi^n_j$. By H\"{o}lder's inequality, $\left\{\dd \gamma_{n,1} \wedge \ldots \wedge \dd\gamma_{n, j-1} \wedge \xi^n_j \wedge \dd\gamma_{n,j+1} \wedge \ldots \wedge \dd \gamma_{n, L-1}\right\}$ converges weakly in $L^{\mathcal{Q}}$, with the index
\begin{align*}
\mathcal{Q} &= \left[ \left(\frac{1}{p_1} + \ldots + \frac{1}{p_{j-1}} + \frac{1}{p_{j+1}} + \ldots +\frac{1}{p_{L-1}}\right) + \frac{1}{q_j'} \right]^{-1}\\
&= \left[\left(\frac{1}{q_j} - \frac{1}{p_L}  \right)+ \frac{1}{q_j'}\right]^{-1}\\
&= \left[1-\frac{1}{p_L}\right]^{-1} \\
&= p_L'.
\end{align*}

Thus, \eqref{conv, true} is proved, and hence follows the assertion~\eqref{assert}.

\noindent
{\bf Step~2C.} Let us return to the  term $\dd\gamma_{n,1} \wedge \ldots \wedge \dd\gamma_{n,L-1}$. Recall from the fake claim~\eqref{false} and the ensuing paragraph that it cannot be proved to converge strongly in $L^{p_L'}$. However,  
\begin{equation}\label{dd}
\dd \left[\dd\gamma_{n,1} \wedge \ldots \wedge \dd\gamma_{n,L-1}\right] = 0\qquad \text{in the sense of distributions.}
\end{equation}
Once established, this identity together with \eqref{assert} above implies \eqref{strong conv, to prove}.

 The key to the justification of  \eqref{dd} lies in its definition: the term $\left[\dd\gamma_{n,1} \wedge \ldots \wedge \dd\gamma_{n,L-1}\right]$ therein is well defined only when understood as follows:
\begin{align}\label{ddd}
\left[\dd\gamma_{n,1} \wedge \ldots \wedge \dd\gamma_{n,L-1}\right] &:= \dd \left\{\gamma_{n,1} \wedge \dd\gamma_{n,2} \wedge \ldots \wedge \dd\gamma_{n,L-1}\right\}\nonumber\\
&\qquad\qquad \text{in the sense of distributions}.
\end{align}
Using \eqref{ddd} and that $\dd\circ\dd =0$ in the sense of distributions, we obtain \eqref{dd} immediately.

\noindent
{\bf Step~2D.} It remains to justify \eqref{ddd}. The crucial observation is as in Step~1 above: since  only $(L-1)$ forms are taken into account, we are in the subcritical regime; that is, $\sum_{i=1}^{L-1} \frac{1}{p_i} < 1+\frac{1}{N}$. In the sequel, let us show that $\left\{\gamma_{n,1} \wedge \dd\gamma_{n,2} \wedge \ldots \wedge \dd\gamma_{n,L-1}\right\}$ is weakly convergent in $L^1$.

For this purpose, note that $\left\{\dd \gamma_{n,j}\right\}_{n \in \mathbb{N}}$ is weakly convergent in $L^{p_j}$ for each $j \in \{2, \ldots, L-1\}$, while $\left\{\gamma_{n,j}\right\}$ is weakly convergent in $W^{1,p_j}$, hence strongly convergent in $L^r$ for any $r \in \left[1, p_j^\star = \frac{Np_j}{N-p_j}\right[$. One may take $r=q_j' - \e$ for arbitrary $\e>0$. In fact, for each $j \in \{1,2,\ldots,L\}$,
\begin{align*}
q_j' = \frac{1}{1-\frac{1}{q_j}}
= \frac{1}{1- \left(\sum_{i=1}^L \frac{1}{p_i}\right) + \frac{1}{p_j}} \leq  \frac{1}{\frac{1}{p_j} - \frac{1}{N}} = p_j^\star.
\end{align*} 

As a consequence, $\left\{\gamma_{n,1} \wedge \dd\gamma_{n,2} \wedge \ldots \wedge \dd\gamma_{n,L-1}\right\}$ converges weakly in the $(L-1)$-fold product of Lebesgue spaces $\left(L^{p_2}\cdot \ldots \cdot L^{p_{L-1}} \cdot L^{q_1' - \e}\right)$. The indices thereof satisfy 
\begin{align*}
\frac{1}{\frac{1}{p_2} + \ldots + \frac{1}{p_{L-1}} + \frac{1}{q_1'-\e}} = \frac{1}{\frac{1}{q_1} - \frac{1}{p_L} + \frac{1}{q_1' - \e}}  = \frac{1}{1-\frac{1}{p_L}} - \kappa = p_{L}' -\kappa\quad \text{for } \kappa = \mathfrak{o}(\e),
\end{align*}
where $\kappa = \mathfrak{o}(\e)$ means that $\kappa$ can be made arbitrarily small by taking $\e >0$ sufficiently close to zero. But $p_L<\infty$, so by choosing $\e>0$ sufficiently small and applying the H\"{o}lder's inequality, we can require that $p_L'-\kappa \geq 1$. This choice warrants the  convergence of $\left\{\gamma_{n,1} \wedge \dd\gamma_{n,2} \wedge \ldots \wedge \dd\gamma_{n,L-1}\right\}$ in the weak $L^1$-topology. 

Thus we arrive at \eqref{ddd}. The proof is now complete.  \end{proof}

\subsection{The div-curl formulation} The prototype of the wedge product theorem of Robbin--Rogers--Temple \cite{key1}, hence of Theorem~\ref{thm: wedge} in this paper, is the classical \emph{div-curl lemma} of Murat and Tartar  (\cite{mur1, mur2, tar1, tar2}). An analogous div-curl-type result follows easily from Theorem~\ref{thm: wedge}.

\begin{corollary}\label{cor: div-curl}
Let $\left(X^N,g\right)$ be an $N$-dimensional closed Riemannian manifold; $N \geq 2$. Let $1<p,q<\infty$ be such that
\begin{equation*}
1 \leq \frac{1}{p} + \frac{1}{q} \leq 1+\frac{1}{N}.
\end{equation*}
Consider two sequences of differential $\ell$-forms $\left\{ \alpha^n \right\}\subset L^p\left(X, \bigwedge^{\ell} T^*X \right)$ and $\left\{ \theta^n \right\}\subset L^q\left(X, \bigwedge^{\ell} T^*X \right)$ satisfying the following conditions:
\begin{enumerate}
\item
$\alpha^n \weak \bara$ weakly in $L^p$;
\item
$\theta^n \weak \bart$ weakly in $L^q$;
\item
$\left| \alpha^n - \bara \right|^p\dvg \weak \mu$ weakly-$\star$ in $\M$;
\item
$\left| \theta^n - \bart \right|^q\dvg \weak \nu$ weakly-$\star$ in $\M$;
\item
$\dd\alpha^n \to \dd\bara$ strongly in $W^{-1,q'}$;
\item
$\dd^* \theta^n \to \dd^*\bart$ strongly in $W^{-1,p'}$. 
\end{enumerate}
Then we have the following convergence (modulo subsequences) in the sense of  distributions:
\begin{equation*}
\bra \alpha^n, \theta^n\ket_g \longrightarrow  \bra\bara, \bart\ket_g + \sum_{k=1}^\infty {\rm div}_g\left(r^k\delta_{x^k}\right) \qquad \text{in } \mathcal{D}'
\end{equation*}
for some sequences of points $\left\{x^k\right\} \subset X$ and vectors $\left\{r^k\right\}$ such that
\begin{align*}
\left|r^k\right| \lesssim_{p,q,N,X} \left[\mu\left(\left\{x^k\right\}\right)\right]^{\frac{1}{p}}\left[\nu\left(\left\{x^k\right\}\right)\right]^{\frac{1}{q}}.
\end{align*} 

Moreover, if $\frac{1}{p}+\frac{1}{q}<1+\frac{1}{N}$, then all $r^k$ are zero.
 \end{corollary}

Recall that $\bra\alpha, \theta\ket_g := \star\left[\alpha \wedge \star \theta\right]$ for differential $\ell$-forms  $\alpha$ and $\theta$ on $X$.

When $\ell=1$ (with differential one-forms identified with vectorfields via the musical isomorphisms), Corollary~\ref{cor: div-curl} reduces to \cite[Theorem~2.3]{bcm} in Briane--Casado-D\'{i}az--Murat. The Hodge star $\star$,  the codifferential $\dd^* = (-1)^{N(\ell+1)+1}\star \dd\star$, the divergence ${\rm div}_g$, and the inner product $\bra \bullet, \bullet\ket_g$ are all taken with respect to the Riemannian metric $g$. Meanwhile, as commented beneath the statement of Theorem~\ref{thm: wedge}, when $\frac{1}{p} + \frac{1}{q} <1$, classical div-curl lemma or wedge product theorem applies to yield that $\bra\alpha^n,\theta^n\ket_g \to \bra\bara,\bart\ket_g$ in the sense of distributions.

\begin{proof}[Proof of Corollary~\ref{cor: div-curl}]
Applying Theorem~\ref{thm: wedge} to $\alpha^n$ and $\beta^n := \star \theta^n \in L^q\left(X,\bigwedge^{N-\ell}T^*X\right)$, one obtains that
\begin{align*}
\alpha^n \wedge  \beta^n	 \longrightarrow \bara \wedge \barb + \sum_{k=1}^\infty \dd  \left(v^k\delta_{x^k}\right) \qquad \text{ in } \mathcal{D}'\left(X, \bigwedge^{N}T^*X\right),
\end{align*}
where $v_k$ are $(N-1)$-covectors. We conclude the proof by taking the Hodge star on both sides and identifying $r^k$ with $\star v^k$ via the musical isomorphism $ \sharp: TX \cong T^*X$, using the fact that both Hodge star and $\sharp$ are Sobolev space isometries. \end{proof}

\subsection{Smooth cycles} We observe from Theorem~\ref{thm: wedge} that, even in the critical case $\frac{1}{p} + \frac{1}{q} = \frac{1}{N} + 1$, the concentration distribution is ``exact''. Thus we make the following observation (where $\langle\bullet,\bullet\rangle$ denotes the duality paring between currents and differential forms):
\begin{corollary}
Let $\left(X^N,g\right)$ be an $N$-dimensional closed Riemannian manifold; $N \geq 2$. Let $1<p,q<\infty$ be such that $\frac{1}{p} + \frac{1}{q} = \frac{1}{N} + 1$. Consider differential forms $\{\alpha^n\}$ and $\{\beta^n\}$ as in Theorem~\ref{thm: wedge}, where $\ell_1 + \ell_2 = N = \dim X$. Then for any smooth  N-current $T \in \mathcal{R}^N(X)$ which is a cycle (\emph{i.e.}, $\p T = 0$), we have $\bra T, \alpha^n \wedge \beta^n \ket \to \bra T, \bara \wedge \barb\ket$ as $n \to \infty$. 
\end{corollary}  

\begin{proof}
This follows from $\bra T, \dd\left(\sum_k v^k\delta_{x^k}\right)\ket = \bra \p T, \sum_k v^k\delta_{x^k}\ket = 0$, where $ \dd\left(\sum_k v^k\delta_{x^k}\right)$   is the concentration distribution as in Theorem~\ref{thm: wedge}. Alternatively, rectifiable currents  with $\CC^\infty$-regularity can be represented by smooth vectorfields, for which the cycle condition $\p T=0$ is equivalent to the divergence-free condition. So the corollary also follows from $\dd^*\circ \dd^* = 0.$  \end{proof}

\section{The ``endpoint critical'' cases $\{p,q\} = \{1,N\}$}\label{sec: endpoint critical case}

In this section we discuss the generalisation of \S\S 3 $\&$ 4 in \cite{bcm}, \emph{i.e.}, the cases $(p,q) = (1,N)$ or $(p,q) = (N,1)$, to differential forms on closed manifolds. These are the endpoint cases of the critical regime $\frac{1}{p}+ \frac{1}{q} = \frac{1}{N}+1$, for which we need stronger compactness assumptions for $\dd\alpha^n$ and $\dd\beta^n$ (namely, Assumptions~(5)+(6) in Theorem~\ref{thm: endpoint}) to ensure the distributional convergence of wedge products. 

For the convenience of readers, let us reproduce Theorem~\ref{thm: endpoint} below.

\begin{theorem*}
Let $\left(X^N,g\right)$ be an $N$-dimensional closed Riemannian manifold; $N \geq 2$. 
Consider two sequences of differential forms $\left\{ \alpha^n \right\}\subset \M\left(X, \bigwedge^{\ell_1} T^*X \right)$ and $\left\{ \beta^n \right\}\subset L^N\left(X, \bigwedge^{\ell_2} T^*X \right)$ satisfying the following conditions:
\begin{enumerate}
\item
$\alpha^n \weak \bara$ weakly-$\star$ in $\M$;
\item
$\beta^n \weak \barb$ weakly in $L^N$;
\item
$\left| \alpha^n - \bara \right| \weak \mu$ weakly-$\star$ in $\M$;
\item
$\left| \beta^n - \barb \right|^N  \weak \nu$ weakly-$\star$ in $\M$;
\item
$\dd\alpha^n \to \dd \bara$ strongly in $W^{-1,N'}$;
\item
$\dd\beta^n \to \dd \barb$ strongly in $L^N$.
\end{enumerate}
Then we have the following convergence (modulo subsequences) in the sense of  distributions:
\begin{equation}
\alpha^n \wedge \beta^n \longrightarrow  \bara \wedge \barb + \sum_{k=1}^\infty \dd \left(v^k\delta_{x^k}\right) \qquad \text{in } \mathcal{D}'%\left(X, \bigwedge^{\ell_1+\ell_2-1}T^*X\right),
\end{equation}
for some sequences of points $\left\{x^k\right\} \subset X$ and $(\ell_1+\ell_2-1)$-covectors $\left\{v^k\right\}$. Moreover,
\begin{equation*}
\left|v^k\right| \lesssim_{N,X}
\left[\mu\left(\left\{x^k\right\}\right)\right]\left[\nu\left(\left\{x^k\right\}\right)\right]^{\frac{1}{N}}.
\end{equation*}
\end{theorem*}

Recall that $\M\left(X, \bigwedge^\bullet T^*X\right)$ is understood as the space of $\bigwedge^\bullet T^*X$-valued finite Radon measures on $X$, where $\bigwedge^\bullet T^*X$ is a vector space. Thus the total variation measure (hence the total variation norm) of each of its elements is well defined. Also note that $N'=\frac{N}{N-1}$.

Our proof essentially follows the strategies in \cite[\S 3]{bcm}, and the idea is similar to that of Theorem~\ref{thm: wedge}, except that more delicate analysis for endpoint Sobolev embeddings are required. To avoid redundancies, we shall only elaborate on the necessary modifications.

\begin{proof}[Proof of Theorem~\ref{thm: endpoint}] We divide our arguments into six steps below.

\smallskip
\noindent
{\bf Step 1.} As for Theorem~\ref{thm: wedge}, consider the decomposition
\begin{equation*}
\begin{cases}
\alpha^n = \dd\gamma_n + \xi^n,\\
\beta^n = \dd\zeta^n +\eta^n.
\end{cases}
\end{equation*}
The proof for Proposition~\ref{prop: decomposition of diff forms in wedge prod} yields that
\begin{align*}
&\zeta^n \weak \zeta  \text{ weakly in } W^{1,N}\left(X,\bigwedge^{\ell_2-1}T^*X\right),\\
&\eta^n \longrightarrow \eta \text{ strongly in } W^{1,N}\left(X,\bigwedge^{\ell_2}T^*X\right),\\
& \barb = \dd\zeta + \eta.
\end{align*}

Estimates for $\gamma_n$ and $\xi^n$ pertain to the endpoint elliptic estimates \emph{\`{a} la} Bourgain--Brezis \cite{bb04} and Brezis--Van Schaftingen \cite{bv}; see Appendix~\ref{sec: appendix b}. One may choose $\gamma_n$ and $\xi^n$ such that
\begin{align*}
&\gamma_n \longrightarrow \gamma  \text{ strongly in } W^{1,N'}\left(X,\bigwedge^{\ell_1-1}T^*X\right),\\
&\xi^n \weak \xi \text{ weakly in } W^{-1,N'}\left(X,\bigwedge^{\ell_1}T^*X\right) \text{ and weakly-$\star$ in  }\M \left(X,\bigwedge^{\ell_1}T^*X\right),\\
& \bara = \dd\gamma + \xi.
\end{align*}

\smallskip
\noindent
{\bf Step 2.} Now let us define
\begin{align*}
\alpha^n \wedge \beta^n &:= \dd \gamma_n \wedge \dd\zeta^n  + \xi^n \wedge \dd\zeta^n + \dd\gamma_n \wedge \eta^n + \xi^n \wedge \eta^n \nonumber\\
&= J_1^n + J_2^n + J_3^n + J_4^n \qquad \text{in the sense of distributions}.
\end{align*}
Indeed, $J^n_3$ is a strong-strong pairing of $L^{N'}$ and $W^{1,N}$, while $J^n_4$ is a weak-strong pairing of $W^{-1,N'}$ and $W^{1,N}$; hence, they converge in the sense of distributions. Also, for $J_1^n$ one observes that $\gamma_n \wedge \dd\zeta^n$ is a weak-strong pairing of $L^N$ and $L^{N'}$, so it converges in the weak $L^1$-topology. As a consequence, we may define $\dd\gamma_n \wedge \dd\zeta^n := \dd \left(\gamma_n \wedge \dd\zeta^n\right)$ in the distributional sense.

It remains to investigate the (well-definedness as a distribution and) convergence of the weak-weak pairing $J_2^n$.

\smallskip
\noindent
{\bf Step 3.} To see that $$J_2^n:=
\xi^n \wedge \dd\zeta^n$$ is well defined as a differential form-valued distribution on $X$, one needs to prove that for any testform $\varrho \in \CC^\infty\left(X, \bigwedge^{N-\ell_1-\ell_2} T^*X\right)$, the following expression is well defined:
\begin{equation*}
I^n \equiv \int_X \xi^n \wedge \dd\zeta^n \wedge \varrho.
\end{equation*}
Clearly there is nothing to prove when $\ell_1 + \ell_2 > N$.

For this purpose, we integrate by parts  using the superdistributivity of $\dd$ over wedge product and the Stokes' theorem to define
\begin{align*}
I^n &= I_1^n + I_2^n\\
&:= \left[(-1)^{\ell_1+1}\int_X \dd\xi^n \wedge \zeta^n \wedge \varrho\right] + \left[(-1)^{\ell_2+1} \int_X\xi^n \wedge \zeta^n \wedge \dd\varrho\right].
\end{align*}

 For $I^n_1$, we make use of the decomposition $\alpha^n = \dd\gamma_n + \xi^n$ to infer that $\dd\alpha^n = \dd\xi^n$ in the sense of distributions. But it is assumed that $\dd\alpha^n \to \dd \bara$ strongly in $W^{-1,N'}$ (see assumption~(5)), so the term $\dd\xi^n \wedge \zeta^n$ appearing in the integrand of $I^n_1$ is a strong-weak pairing of $W^{-1,N'}$ and $W^{1,N}$, thus making $I^n_1$ well defined and converge as $n \to \infty$ to $(-1)^{\ell_1 + 1} \int_X \dd\xi \wedge \zeta \wedge \varrho$ as desired.

It now remains to investigate the convergence of
\begin{equation*}
\tilde{I}^n_2:=\int_X\xi^n \wedge \zeta^n \wedge \dd\varrho,
\end{equation*}
which is nothing but $I_2^n$ with the immaterial sign neglected. Here $\varrho$ is an arbitrary testform in $\CC^\infty\left(X, \bigwedge^{N-\ell_1-\ell_2} T^*X\right)$, and $\xi^n \wedge \zeta^n$ is a  weak-weak pairing of $W^{-1,N'}$ and $W^{1,N}$.

\smallskip
\noindent
{\bf Step 4.} To make sense of $\xi^n \wedge \zeta^n$,  recall that $\xi^n$ is chosen to be coexact: by Hodge decomposition $\xi^n = \dd^* k^n + h^n$ where $h^n$ is harmonic. Then one may find from Lemma~\ref{lem: brezis-vs} that $$\sigma^n = \Delta^{-1}\xi^n \in W^{1,N'}\left(X;\bigwedge^{\ell_1}T^*X\right),$$ which is unique by fixing the cohomology class once and for all. Moreover, we have $\dd^*\sigma^n=0$. 

The above arguments lead to
\begin{align}\label{sigma, step4}
\tilde{I}^n_2 &= \int_X \Delta \sigma^n \wedge \zeta^n \wedge \dd\varrho \\ 
&= \int_X \left(\dd\dd^*\sigma^n+\dd^*\dd \sigma^n\right) \wedge \zeta^n \wedge \dd\varrho\\
&= \int_X \dd\sigma^n \wedge \dd^*\left[\zeta^n\wedge \dd\varrho\right]. 
\end{align}
In the right-most term, $\left\{\dd\sigma^n\right\}$ is weakly convergent in $L^{N'}$, and $\left\{\dd^*\left[\zeta^n\wedge \dd\varrho\right]\right\}$ is weakly convergent in $L^{N}$, after passing to subsequences if necessary. Denote the weak limits of $\left\{\sigma^n\right\}$ in $W^{1,N'}$ (existence follows from the quantitative statement in Lemma~\ref{lem: brezis-vs}) and $\left\{\zeta^n\right\}$ in $W^{1,N}$ by $\sigma$ and $\zeta$, respectively.   We are now in the situation of applying Lemma~\ref{lem: appendix} to obtain a defect measure $\varpi = \varpi[\varrho] \in \M\left(X;\bigwedge^N T^*X \right)$ that is linear in the testform $\varrho$:
\begin{align*}
\lim_{n \to \infty} \tilde{I}^n_2 = \int_X \dd\sigma \wedge \dd^*\left[\zeta\wedge\dd\varrho\right] +  \varpi (X).
\end{align*}
By Stokes' theorem and the fact that $\dd^*\sigma=0$ in the sense of distributions, we can rewrite this identity as follows:
\begin{align*}
\lim_{n \to \infty} \tilde{I}^n_2 &= \int_X \Delta \sigma \wedge \zeta \wedge \dd\varrho + \varpi (X) \\&= \int_X \xi \wedge \zeta \wedge \dd\varrho + \varpi (X). 
\end{align*}

It thus remains to investigate the defect measure(-valued differential $N$-form) $\varpi$.

\smallskip
\noindent
{\bf Step 5.} Consider the limiting measures 
\begin{align*}
&\lambda := \M-\lim_{n\to\infty} \Big|\dd^*\big[\zeta^n \wedge \dd\varrho\big] - \dd^*\big[\zeta \wedge \dd\varrho\big] \Big|^N,\\
&\lambda' := \M-\lim_{n\to\infty} \Big|\dd \left(\sigma^n-\sigma\right) \Big|^{N'}.
\end{align*}
We \emph{claim} that the total variation of $\lambda$ is controlled by the measure $\nu$ in assumption~(4). That is,
\begin{equation}\label{claim, lambda and nu}
|\lambda|(B) \lesssim _{(X,g), \|\varrho\|_{\CC^2}} \nu(B)\qquad \text{ for any Borel subset } B \subset X.
\end{equation}

\begin{proof}[Proof of the claim~\eqref{claim, lambda and nu}]
It suffices to prove for the case that $B = {\bf B}_1(0) \subset \R^N$ and that $X$ is Euclidean flat.

To see this, let $\left\{\chi_1, \ldots, \chi_K\right\}$ be a $\CC^\infty$-partition of unity  on the compact manifold $(X,g)$  subordinate to an atlas of $\CC^\infty$-charts  $\left\{O_1, \ldots, O_K\right\}$, each of which is diffeomorphic to ${\bf B}_1(0)$. Then, viewing $|\bullet|^N$ as positive measures on $X$, for any Borel subset $B\subset X$ we have that
\begin{align*}
\Big|\dd^*\big[\zeta^n \wedge \dd\varrho\big] - \dd^*\big[\zeta \wedge \dd\varrho\big] \Big|^N(B)&=\left|\dd^*\left[\left(\zeta^n - \zeta\right) \wedge \dd\varrho\left(\sum_{i=1}^K \chi_i\right)\right] \right|^N(B)\\
&\lesssim_{N,K} \sum_{i=1}^K \Big| \dd^* \left[\chi_i(\zeta^n-\zeta) \wedge \dd\varrho\right] \Big|^N(B)\\
&\lesssim_{K,N,(X,g)} \left\|\varrho\right\|_{\CC^2}^N\cdot\left\{\sum_{i=1}^K \Big|\na \left(\chi_i\zeta^n-\chi_i\zeta\right)\Big|^N(B)\right\}.
\end{align*}
The constant involved in the last inequality depends on the $\CC^1$-geometry of $(X,g)$. One may thus infer from the Leibniz rule, the convergence $|\zeta^n-\zeta|^N \to 0$ as $n \to \infty$, and the defining properties of the partition of unity $\{\chi_i\}_1^K$ that 
\begin{align*}
\Big|\dd^*\big[\zeta^n \wedge \dd\varrho\big] - \dd^*\big[\zeta \wedge \dd\varrho\big] \Big|^N (B) &\lesssim_{(X,g),\left\|\varrho\right\|_{\CC^2}} \left\{\sum_{i=1}^K \Big|\chi_i\na \left(\zeta^n-\zeta\right)\Big|^N\left(B\cap O_i\right)\right\} + \mathfrak{o}_n(1),
\end{align*}
where $\lim_{n\to\infty}\mathfrak{o}_n(1) = 0$. Thus, it is enough to prove for $\left(\chi_i\zeta^n,\chi_i\zeta\right)$ in lieu of $\left(\zeta^n,\zeta\right)$ in \eqref{claim, lambda and nu}; or, equivalently, we may assume that $B$ is contained in one single chart $O_i$.

%Also, as $\zeta^n-\zeta$ is coexact, by the same reasoning as in the previous paragraph, one may suppose that $\chi_i\zeta^n-\chi_i\zeta$ is coexact. 

A similar computation for commuting $\na$ with the diffeomorphism between the chart containing the support of $\chi_i$ and the Euclidean ball ${\bf B}_1(0)$ shows that \eqref{claim, lambda and nu} is invariant under such diffeomorphisms. Hence, we may assume without loss of generality that $B={\bf B}_1(0)\subset \R^N$ and that $\chi_i\zeta^n-\chi_i\zeta$ is coexact and compactly supported in $B$ for each fixed $i \in \{1, \ldots, K\}$. We relabel $\left(\chi_i\zeta^n,\chi_i\zeta\right)$ in the above as $\left(\zeta^n,\zeta\right)$ from now on. %In particular, assume that  $\chi_i\zeta^n-\chi_i\zeta$ is coexact.

It remains to argue, in view of the previous reduction, that 
\begin{align}\label{gaffney in RN}
\Big|\na\left(\zeta^n-\zeta\right)\Big|^N(B) \lesssim \Big|\dd \left(\zeta^n-\zeta\right)\Big|^N(B). 
\end{align}
Assuming \eqref{gaffney in RN} and recalling from Step~1 that $\dd \left(\zeta^n-\zeta\right) = \left(\beta^n-\barb\right) - \left(\eta^n-\eta\right)$, where $\eta^n \to \eta$ strongly in $W^{1,N}$ and $\left|\beta^n-\barb\right|^N\weak \nu$ in $\M$ by Assumption~(4), we may conclude the \emph{claim}. 

To see \eqref{gaffney in RN}, we consider a 
standard mollifier $\mathcal{J}_\delta$ acting on $\left|\na\left(\zeta^n-\zeta\right)\right|^N$, which is well defined since $\zeta^n-\zeta$ is compactly supported in $B = {\bf B}_1(0) \subset \R^N$. Then for $$s^n:=\na\left(\zeta^n-\zeta\right)$$ we have that
\begin{align*}
\Big|\left\|s^n\right\|_{L^N(B)} - \left\|s^n\star\mathcal{J}_\delta\right\|_{L^N(B)}\Big| &\leq \left\|s^n-s^n \star \mathcal{J}_\delta\right\|_{L^N(B)} \lesssim \mathfrak{o}_n(1) \quad \text{ as } \delta \to 0^+. 
\end{align*}
So one can further assume $\zeta^n, \zeta \in \CC^\infty$. In this case, since $\dd^*(\zeta^n-\zeta)=0$ and $\zeta^n-\zeta \in W^{1,N}_0\left(B;\bigwedge^{\ell_2-1} \R^N\right)$ (so that the usual Poincar\'{e}'s inequality applies), we deduce \eqref{gaffney in RN} directly from the Gaffney's inequality. See \cite{gaff1, gaff2, gaff3, gaff4} among other references.  

The proof of the \emph{claim} is now complete.   
\end{proof}

\smallskip
\noindent
{\bf Step 6.} As argued in \S\ref{subsec, where conc comp is used}, it  remains to prove the bound
\begin{equation}\label{LN'-->L1 bd}
\left\{ \int_X|\varphi|^{N'}\,\dd\lambda' \right\}^{\frac{1}{N'}}   \lesssim_{N,X}   \int_X|\varphi|\,\dd \mu \quad \text{for any } \varphi \in \CC^\infty(X).
\end{equation}
Once this is established, an application of the second concentration compactness lemma \emph{\`{a} la} P.-L. Lions \cite[p.158, Lemma~I.1]{lions} together with Lemma~\ref{lem: appendix} readily concludes the proof. See \eqref{conc comp} and the ensuing arguments, as well as \cite[p.486, from Equation~(50) to the end of the proof of Theorem~3.1]{bcm}.

To this end, observe that
\begin{align}\label{Mlim}
\M-\lim_{n\to\infty}\left| \Delta \left(\sigma^n-\sigma\right)\right| &= \M-\lim_{n\to\infty}\left|\xi^n-\xi\right| \nonumber\\
&= \M-\lim_{n\to\infty}\left|\left(\alpha^n - \dd\gamma_n\right) - \left(\alpha-\dd\gamma\right)\right| \nonumber\\
&= \M-\lim_{n\to\infty} \left|\alpha^n-\alpha\right| \nonumber\\
&= \mu.
\end{align}
The penultimate equality follows from  the strong convergence $\dd\gamma_n \to \dd\gamma$ in $L^{N'}$, and the final one from the definition of $\mu$. Also, as $\sigma^n$ is coexact, so is $\Delta \left(\sigma^n-\sigma\right)$. By \eqref{Mlim} and  $\lambda' := \M-\lim_{n\to\infty} \left|\dd \left(\sigma^n-\sigma\right) \right|^{N'}$, we see that \eqref{LN'-->L1 bd} follows immediately from the bound
\begin{align}\label{LN'-->L1 bd, eqvt}
&\limsup_{n \to \infty} \left\| \varphi\, \dd\left(\sigma^n-\sigma\right) \right\|_{L^{N'}(X)} \nonumber\\
&\qquad\qquad\lesssim_{N,X} \limsup_{n \to \infty} \left\|\varphi\, \Delta \left(\sigma^n-\sigma\right)\right\|_{\M(X)}\quad \text{for any } \varphi \in \CC^\infty(X).
\end{align}
In contrast to \eqref{LN'-->L1 bd}, the norms $\|\bullet\|_{L^{N'}(X)}$ and $\|\bullet\|_{\M(X)}$ are taken with respect to the Riemannian metric $g$ as usual.

To prove \eqref{LN'-->L1 bd, eqvt}, notice that it is equivalent to 
\begin{align}\label{LN'-->L1 bd, eqvt2}
\limsup_{n \to \infty} \left\|\dd \left\{\varphi\left(\sigma^n-\sigma\right)\right\} \right\|_{L^{N'}(X)} &\lesssim_{N,X} \limsup_{n \to \infty} \left\| \Delta\left\{\varphi\left(\sigma^n-\sigma\right)\right\}\right\|_{\M(X)},
\end{align}
in view of the identities 
\begin{align*}
\dd \left[\varphi\left(\sigma^n-\sigma\right)\right] = \dd\varphi \wedge\left(\sigma^n-\sigma\right) +  \varphi\, \dd\left(\sigma^n-\sigma\right),
\end{align*}
where $\dd\varphi \wedge\left(\sigma^n-\sigma\right) \to 0$ in $L^{N'}$, and 
\begin{align*}
 \Delta\left\{\varphi\left(\sigma^n-\sigma\right)\right\} = \left(\Delta\varphi\right)\left(\sigma^n-\sigma\right) + 2g^{ij} \na_i\varphi\na_j\left(\sigma^n-\sigma\right) + \varphi \Delta \left(\sigma^n-\sigma\right),
\end{align*}
where $\left(\Delta\varphi\right)\left(\sigma^n-\sigma\right) \to 0$ in $L^{N'}$ and $g^{ij} \na_i\varphi\na_j\left(\sigma^n-\sigma\right) \weak 0$ weakly-$\star$ in $\M$.

It remains to justify \eqref{LN'-->L1 bd, eqvt2}. For the moment, let us work with \emph{an additional assumption}: 
\begin{align}\label{additional assumption}
\text{$X$ has no nontrivial harmonic $\ell_1$-form, namely that $H_{\rm dR}^{\ell_1}(X)=\{0\}$}.
\end{align}
It then follows from the endpoint elliptic estimate in Lemma~\ref{lem: brezis-vs} that
\begin{align}\label{zzz}
&\left\|\dd \left\{\varphi\left(\sigma^n-\sigma\right)\right\} \right\|_{L^{N'}(X)} \lesssim_{N,X} \left\| \Delta\left\{\varphi\left(\sigma^n-\sigma\right)\right\}\right\|_{\M(X)} + \left\|\dd^*\Delta\left\{\varphi\left(\sigma^n-\sigma\right)\right\}\right\|_{W^{-2,N'}(X)}.
\end{align}
The last term on the right-hand side vanishes under $\limsup_{n \to \infty}$, because
\begin{align*}
\dd^*\Delta\left\{\varphi\left(\sigma^n-\sigma\right)\right\} &= \Delta\dd^*\left\{\varphi\left(\sigma^n-\sigma\right)\right\}\\
&= (-1)^{N(\ell_1+1)+1} \Delta\star \dd\left\{ \varphi \left[\star\left(\sigma^n-\sigma\right)\right]\right\} \\
&= (-1)^{N(\ell_1+1)+1}  \Delta \star\Big\{ \dd\varphi \wedge \left[\star\left(\sigma^n-\sigma\right)\right] \Big\} + \Delta \left[\varphi\, \dd^*\left(\sigma^n-\sigma\right)\right] \\
&=: A_1+A_2,
\end{align*}
thanks to the commutativity between $\dd^*$ and $\Delta$, the definition of codifferential $\dd^*$, and the superdistributivity of $\dd$ under wedge product. The term $A_2$ is zero in $W^{-2,N'}$ since $\sigma^n-\sigma \in W^{1,N'}$ is coexact, while $\lim_{n \to \infty}\|A_1\|_{W^{-2,N'}(X)}=0$ since $\sigma^n - \sigma$ converges weakly to zero in $W^{1,N'}$, and hence strongly in $L^{N'}$ by Rellich's lemma.

Thus, \eqref{LN'-->L1 bd, eqvt2} is established under the hypothesis \eqref{additional assumption}, from which \eqref{LN'-->L1 bd} follows. But by virtue of a standard partition of unity argument, the assertion we are after is essentially local; \emph{i.e.}, it suffices to prove for the case that $\alpha^n$ and $\beta^n$ are compactly supported in one chart $O_i$. See the proof for \emph{claim}~\eqref{claim, lambda and nu} in Step~5 above for the notations and details of arguments. Thus the assumption in \eqref{additional assumption} can be safely removed.

The proof of Theorem~\ref{thm: endpoint} is now complete.   \end{proof}

\section{Application to Gauss--Codazzi-Ricci equations for isometric immersions}\label{sec: isom imm}

Our main result of this section is the following, which generalises Theorem~\ref{thm: isom imm, 2D} in the Introduction \S\ref{sec: intro} to arbitrary dimensions and codimensions.

\begin{theorem}[Weak continuity of Gauss--Codazzi--Ricci equations]\label{thm: isom imm}
Let $\left(X,g\right)$ be a Riemannian immersed submanifold of dimension $N \geq 2$ in $\R^{N+k}$. Consider a family $\left\{\sol^\e = (\two^\e, \na^{\perp,\e})\right\}$ of weak solutions to the Gauss--Codazzi--Ricci equations that converges to $\overline{\sol}$ in the weak $L^p_\loc$-topology; $p \in \left]\frac{2N}{N+1},\infty\right[$. Suppose  the coexact parts of $\left\{\sol^\e\right\}$ are precompact in strong $L^{p'}_\loc$-topology. Then $\overline{\sol}$ is a weak solution to the Gauss--Codazzi--Ricci equations. 
 \end{theorem}
 
 \begin{definition}

For a closed manifold $X$ and any tensor or affine connection $\mathcal{T}$ on it, we set 
\begin{equation}\label{coexact part, leray}
\text{coexact part of } \mathcal{T} = \left({\bf Id} - \dd\Delta^{-1}\dd^*\right) \mathcal{T},
\end{equation}
where $\Delta^{-1}$ is defined with respect to any fixed cohomology group. 
 \end{definition}
 Note that for a vectorfield $V \in \G(TX)$, the expression $\left({\bf Id} - \dd\Delta^{-1}\dd^*\right) V$ modulo musical isomorphisms between $TX$ and $T^*X$ is the well-known \emph{Leray projection} of $V$.
 
 As the main scope of the current paper is on wedge product theorems in compensated compactness theory, we refrain ourselves from giving a detailed exposition on  isometric immersions and/or the Gauss--Codazzi--Ricci equations. %Let us only point out that the Gauss--Codazzi--Ricci equations are compatibility PDEs of curvatures for the existence of isometric immersions. 
 See do Carmo \cite[Chapter~6]{doc} and Tenenblat \cite{ten} for the derivation of these equations, and the monograph \cite{hh} for histories and up-to-date developments of the isometric immersions problem.

Here and hereafter, for an isometric  immersion $\Phi: (X,g) \to (\rnk,\delta)$ (where $\delta$ denotes the Euclidean metric),  the curvature components in the tangential and normal directions of $\Phi$ together constitute the flat Riemann curvature on $\rnk$. This gives rise to the \emph{Gauss--Codazzi--Ricci} equations:
\begin{eqnarray}
&& \delta\big(\two(X,Z), \two (Y,W)\big) - \delta\big( \two(X,W),\two(Y,Z) \big) = R(X,Y,Z,W),\label{gauss}\\
&& \overline{\na}_Y\two(X,Z) - \overline{\na}_X\two(Y,Z)=0,\label{codazzi}\\
&& g\big([\mathcal{S}_\eta, \mathcal{S}_{\zeta}] X, Y\big) = R^E(X,Y,\eta,\zeta),\label{ricci}
\end{eqnarray}
for any $X,Y,Z,W \in \G(TX)$ and $\eta,\zeta \in \G(E)$. Here $[\bullet,\bullet]$ is the commutator, $\delta$ the Euclidean metric, $\mathcal{S}$ the shape operator (equivalent to $\two$ modulo contractions by $g$), $R$ and $R^E$ the Riemann curvature tensors for $(T\M,g)$ and $\left(E,g^E\right)$, respectively, and $\overline{\na}$ the Levi-Civita connection on $\R^{n+k}$. For the moment we take $E = T^\perp\M := T\rnk \slash T[\Phi(\M)]$, the normal bundle of $\Phi$. In general, without knowing \emph{a priori} that an isometric immersion exists, one cannot talk about the normal bundle. Nonetheless, the Gauss--Codazzi--Ricci Equations~\eqref{gauss}--\eqref{ricci} above can still be written down for an arbitrary rank-$k$ vector bundle $E$ over $(X,g)$, equipped with a bundle metric $g^E$ that is compatible with $g$. In this case, $R^E$ is the Riemann curvature tensor defined with respect to $\na^E$, the Levi-Civita connection associated to $g^E$.

The Gauss--Codazzi--Ricci Equations~\eqref{gauss}--\eqref{ricci} are a first-order nonlinear PDE system. In an arbitrary local co-ordinate system $\{x^1, \ldots, x^{n+k}\}$, we set
\begin{equation}\label{h and kappa}
h^\alpha_{ij}:=\delta\big(\two(\p_i,\p_j),\p_\alpha\big)\qquad\text{and}\qquad \kappa^\alpha_{i\beta}:=\delta\big(\na^\perp_{\p_i}\p_\beta,\p_\alpha\big).
\end{equation}
Here and hereafter, we adopt the index convention: 
\begin{equation*}
1\leq i,j,k,\ell,p,q \leq N; \qquad N+1 \leq \alpha,\beta,\gamma \leq N+k;\qquad 1\leq a,b,c \leq N+k.
\end{equation*}
With Einstein's summation convention, the Gauss--Codazzi--Ricci Equations~\eqref{gauss}--\eqref{ricci} are expressed locally as follows:
\begin{equation}\label{gauss, loc}
g_{\alpha\beta}\left(h^\alpha_{ji}h^\beta_{k\ell}-h^\alpha_{ki}h^\beta_{j\ell}\right)=R_{ijk\ell},
\end{equation}
\begin{equation}\label{codazzi, loc}
\frac{\p h^\alpha_{j\ell}}{\p x^k} - \frac{\p h^\alpha_{kj}}{\p x^\ell} = \Gamma^m_{kj} h^\alpha_{\ell p} - \G^p_{\ell j} h^\alpha_{kp} + \kappa^\alpha_{\ell \beta} h^\beta_{kj} - \kappa^\alpha_{k\beta} h^\beta_{\ell j},
\end{equation}
and
\begin{equation}\label{ricci, loc}
\frac{\p\kappa^\alpha_{\ell \beta}}{\p x^k} - \frac{\p \kappa^\alpha_{k\beta}}{\p x^\ell} = g^{pq} \left[ h^\alpha_{p\ell} h^\beta_{kq} - h^\alpha_{pk} h^\beta_{\ell q} \right] - \kappa^\alpha_{k\gamma}\kappa^\gamma_{\ell \beta} + \kappa^\alpha_{\ell \gamma} \kappa^\gamma_{k\beta}.
\end{equation}
Here $R=\{R_{ijkl}\}$ is the Riemann curvature tensor on $(X,g)$.

The Gauss--Codazzi--Ricci equations can be written compactly in E. Cartan's exterior calculus formalism adapted to the Euclidean submanifold theory (abbreviated in the sequel as \emph{the Cartan formalism}), which is a classical topic in differential geometry \cite{spivak}. In local co-ordinates:
\begin{eqnarray}
&&d\omega^i = \sum_j \omega^j \wedge \Omega^i_j;\label{first structure eq}\\
&&0=d\Omega^a_b + \sum_c \Omega^c_b\wedge \Omega^a_c,\label{second structure eq}
\end{eqnarray}
where $\{\omega^i\}_{1\leq i \leq n}$ is the orthonormal coframe on $(T^*X,g)$ dual to $\{\p\slash\p_i\}_{i=1}^n$, and $\{\Omega^a_b\}_{1\leq a,b \leq n+k}$ is defined entry-wise by
\begin{eqnarray}
&& \Omega^i_j (\p_k) := g(\na_{\p_k}\p_i,\p_j);\label{Omega 1}\\
&& \Omega^i_\alpha (\p_j) \equiv -\Omega^\alpha_i (\p_j) := g^E\big(\two(\p_i,\p_j), \eta_\alpha\big);\label{Omega 2}\\
&& \Omega^\alpha_\beta(\p_j):=g^E\big(\na^E_{\p_j}\eta_\alpha,\eta_\beta\big).\label{Omega 3}
\end{eqnarray}
In the above, $\{\p_i\}$ is the orthonormal frame for $(TX,g)$ dual to $\{\omega^i\}$, and $\Omega=\{\Omega^a_b\}$ is called the \emph{connection 1-form}. Note that $\Omega^i_\alpha(\p_j)=h^\alpha_{ij}$ and $\Omega^\alpha_\beta(\p_j)=\kappa^\beta_{i\alpha}$.

The \emph{second structural equation} of the Cartan formalism, namely Equation~\eqref{second structure eq}, is equivalent to the Gauss--Codazzi--Ricci Equations~\eqref{gauss}--\eqref{ricci} as purely algebraic identities, hence in the sense of distributions too. The first structural equation~\eqref{first structure eq}, on the other hand, is equivalent to the torsion-free condition of the Levi-Civita connection $\na$. The connection 1-form takes values in the Lie algebra of antisymmetric matrices. Equations~\eqref{Omega 1}--\eqref{Omega 3} are written compactly as follows (${}^\top$ for transpose):
\begin{align*}
    \Omega = \begin{bmatrix}
        \na & \two\\
        -\two^\top & \na^E
    \end{bmatrix} \in \G \big(X; T^*X\otimes\sonk\big).
\end{align*}
We view the second structural Equation ~\eqref{second structure eq} as an identity on $\G \left(X; \bigwedge^2 T^*X\otimes\sonk\right)$, the space of antisymmetric matrix-valued 2-forms, and commonly denote
\begin{equation*}
    \dd\Omega + [\Omega \wedge \Omega] = 0.
\end{equation*}
Here $\dd$ acts on the differential form factor (\emph{i.e.}, $\bigwedge^\bullet T^*X$) and $[\bullet \wedge \bullet]$ designates the intertwining of wedge product on the form factor and Lie bracket on the matrix factor (\emph{i.e.}, $\sonk$).

\begin{proof}[Proof of Theorem~\ref{thm: isom imm}]
Without loss of generality, we assume that $X$ is compact and hence drop the subscripts ``loc''. Consider the second  structural equation:
\begin{align}\label{structural eq}
\dd\Omega^\e + \Omega^\e \wedge \Omega^\e = 0. 
\end{align}
Here $\left\{\Omega^\e\right\}$ is bounded in $L^p$, as 
\begin{equation}\label{plus}
\Omega^\e = \begin{bmatrix}
\na & \two^\e\\
-\left(\two^\e\right)^\top & \na^{\perp,\e}
\end{bmatrix}
\end{equation}
via the Cartan formalism. One may refer to Clelland \cite{clelland}, Tenenblat \cite{ten}, as well as \cite{cl1, cl2} for details on the Cartan formalism applied to isometric immersions.

By Rellich's lemma, $\left\{\dd\Omega^\e\right\}$ is bounded in $W^{-1,p}$. In fact, we have a stronger mode of convergence --- consider the Hodge decomposition $\Omega^\e = \dd\Upsilon_\e + \Xi^\e$, where $\Xi^\e$ is the divergence-free part that is assumed to be precompact in $L^{p'}$. So $\left\{\dd\Omega^\e\right\} = \left\{\dd\Xi^\e\right\}$ is precompact in $W^{-1,p'}$. 

Now, from the generalised wedge product Theorem~\ref{thm: wedge}, we infer that $$\Omega^\e \wedge \Omega^\e \to \overline{\Omega} \wedge \overline{\Omega}\qquad\text{in the sense of distributions.}$$  Here, as in \S\ref{subsec: weak weak pairing}, the wedge product is defined as follows:
\begin{align*}
\Omega^\e \wedge \Omega^\e = \dd\left(\Upsilon_\e \wedge \dd\Upsilon_\e \right) + \Xi^\e\wedge\Omega^\e + \dd\Upsilon_\e \wedge \Xi^\e,
\end{align*}
and similarly for $\overline{\Omega} \wedge \overline{\Omega}$. The last two terms on the right-hand side are weak-strong pairings of $L^p$ and $L^{p'}$. The first term is also well defined in the sense of distributions ---  $\left\{\Upsilon_\e\right\}$ is bounded in $W^{1,p}$, hence is precompact in $L^{p^\star - \delta}$ for arbitrary $\delta>0$; recall that $p^\star = \frac{Np}{N-p}$ is the Sobolev conjugate of $p$. Thus $\left(\Upsilon_\e \wedge \dd\Upsilon_\e\right)$ is a weak-strong pairing provided that $\frac{1}{p} + \frac{1}{p^\star} < 1$. But this condition is equivalent to $p>\frac{2N}{N+1}$. Note also that we are always in the subcritical regime in the sense of \S\ref{subsec: subcrit}: this is because $\frac{2}{p} < 1+ \frac{1}{N}$ whenever $p>\frac{2N}{N+1}$. 

Therefore, we may pass to the distributional limits separately for $\dd\Omega^\e$ and $\Omega^\e\wedge\Omega^\e$ in \eqref{structural eq}. In this way, one obtains $\dd\overline{\Omega} + \overline{\Omega}\wedge \overline{\Omega} = 0$ in the sense of distributions.  \end{proof}

%Two important remarks are in order:

\begin{remark}\label{remark+}
Our proof above actually shows that the wedge product sequence $\left\{\Omega^\e\wedge\Omega^\e\right\}$ converges in a stronger topology than $\mathcal{D}'$; that is, in the negative Sobolev space $W^{-1,p}$. 

Also, a straightforward adaptation of the proof leads to the following generalisation of Theorem~\ref{thm: isom imm}: one may take $\{\mathfrak{S}^\e\}$ therein to be \emph{approximate (weak)  solutions} instead of weak solutions. That is, the components $\{h^\alpha_{ij}\}$ and $\{\kappa^\alpha_{i\beta}\}$ as in Equation~\eqref{h and kappa}  associated to $\left\{\sol^\e = (\two^\e, \na^{\perp,\e})\right\}$ satisfy the Gauss--Codazzi--Ricci Equations~\eqref{gauss, loc}--\eqref{ricci, loc} with error terms $\mathfrak{o}_1$, $\mathfrak{o}_2$, and $\mathfrak{o}_3$ which converges in $W^{-1,p}_{\rm loc}$ to zero. The study of approximate  solutions to the Gauss--Codazzi--Ricci equations is more interesting in terms of applications; \emph{e.g.}, in establishing the existence theory of solutions to the Gauss--Codazzi--Ricci equations and/or isometric immersions  \cite{csw, csw0}. 
\end{remark}

\begin{remark}
The above theorem and its proof distinguish for the Gauss--Codazzi--Ricci equations one critical exponent $p_{\bf crit} := \frac{2N}{N+1}$ other than the obvious exponent $p_{\bf CS}=2$. (The latter is critical due to the Cauchy--Schwarz inequality.) Observe that
\begin{align}\label{critical exponent}
p_{\bf crit} = \frac{2N}{N+1} < p \leq p_{\bf CS}=2 \leq p' < \frac{2N}{N-1} = p_{\bf crit}^\star \equiv p_{\bf crit}'.
\end{align}

\end{remark}

\bigskip
\appendix
\section{A lemma on weak-weak pairing}

We have the following variant of \cite[Lemma~2.11]{bcm} by Briane--Casado-D\'{i}az--Murat:
\begin{lemma}\label{lem: appendix}
Assume that $1<r<\infty$, and let $E$ be a locally compact Hausdorff topological space. Assume that $\{u_n\} \subset L^r_\loc(E;\C)$ and $\{u'_n\} \subset L^{r'}_\loc(E;\C)$ satisfy $|u_n-u|^r\weak \lambda$ and $|u'_n-u|^{r'} \weak \lambda'$, both weakly-$\star$ in $\M_\loc(E;\C)$. Then, for each quadratic polynomial $\mathbf{q}$, there exists a Radon measure $\varpi \in \M_\loc(E;\C)$ such that 
\begin{align*}
{\bf q}\left(u_n, u_n'\right) \weak {\bf q}\left(u,u'\right) + \varpi \qquad \text{weakly-$\star$ in $E$},
\end{align*}
with the total variation of $\varpi$ satisfying the bound:
\begin{align*}
|\varpi|(B) \lesssim_{\bf q} \left[\lambda(B)\right]^{\frac{1}{r}}\left[\lambda'(B)\right]^{\frac{1}{r'}}\qquad \text{ for any $B \subset E$ Borel}.
\end{align*} 
 
\end{lemma}

\section{Endpoint elliptic estimates}\label{sec: appendix b}

The following result is a variant of \cite[Theorem~3.1]{bv} by Brezis--Van Schaftingen and its generalisation to Radon measures on Euclidean balls in \cite[Proposition~B.1, Appendix B]{bcm}:

\begin{lemma}\label{lem: brezis-vs}
Let $\left(X^N,g\right)$ be a closed Riemannian manifold, and let $\xi \in L^1\left(X,\bigwedge^\ell T^*X\right)$ be a differential form of $L^1$-regularity. Then we can solve for $\sigma \in W^{1,N'}\left(X,\bigwedge^\ell T^*X\right)$ from $\Delta \sigma = \xi$ on $X$, where $\Delta = \dd^*\dd + \dd\dd^*$ is the Laplace--Beltrami operator. The solution is unique modulo harmonic $\ell$-forms, and one has
\begin{align}\label{estimate in Lemma B}
\inf\left\{ \left\|\dd\sigma + h\right\|_{L^{N'}(X)}:\, h \text{ is a harmonic form} \right\} \lesssim_{(X,g)} \|\xi\|_{L^1(X)} + \left\|\dd^*\xi\right\|_{W^{-2,N'}(X)}.
\end{align}

The same result holds for $\xi \in \M\left(X,\bigwedge^\ell T^*X\right)$, with $\|\xi\|_{L^1(X)}$ on the right-hand side of \eqref{estimate in Lemma B} replaced by $\|\xi\|_{\M(X)}$, the total variation norm of $\xi$. Recall that $N'=\frac{N}{N-1}$.
\end{lemma}

Note that \cite[Theorem~3.1]{bv} is originally formulated as an $L^1$-estimate under the  Dirichlet boundary condition on Euclidean domains. By a routine partition of unity  argument one can generalise it to closed manifold $(X,g)$, but the uniqueness of solution can only be retained modulo harmonic $\ell$-forms, which are nontrivial in general on $X$. %This observation explains the necessity of Step~7 in the proof of Theorem~\ref{thm: endpoint}.

For the sake of completeness, we present a proof of Lemma~\ref{lem: brezis-vs} below. We first  establish the following decomposition lemma of the Bourgain--Brezis-type \cite{BB}. In what follows, $\mathcal{H}^\ell(X)$ denotes the space of harmonic $\ell$-forms on $X$.

\begin{lemma}\label{L1}
Let \((X^N,g)\) be a closed Riemannian manifold, $N\geq 2$ and \(1\le \ell\le N-1\).  
For any differential form  
\[
\phi \in W^{1,N}\left(X,\bigwedge^\ell T^*X\right),
\]  
there exist  
\[
\psi \in \left(W^{1,N}\cap L^\infty\right)\left(X,\bigwedge^\ell T^*X\right),\quad 
\eta \in W^{2,N}\left(X,\bigwedge^{\ell-1}T^*X\right),\quad 
h\in \mathcal{H}^\ell(X)
\]  
such that  
\begin{equation}\label{eq:HBB-decomposition}
\phi = \psi + \dd\eta + h,
\end{equation}
and
\begin{equation}\label{eq:HBB-estimate}
\|\psi\|_{W^{1,N}(X)} + \|\psi\|_{L^\infty(X)} + \|\eta\|_{W^{2,N}(X)} + \|h\|_{L^\infty(X)}
\lesssim_{\ell,X} \|\phi\|_{W^{1,N}(X)}.
\end{equation} 
 %= \ker \Delta
\end{lemma}
\begin{proof}
	We first construct the bounded part $\psi$ by localising the differential form and applying the Euclidean Bourgain--Brezis estimate in \cite{BB}. Let $\{U_i\}_{i=1}^m$ be a finite atlas covering $X$, and let $\{\rho_i\}_{i=1}^m$ be a smooth partition of unity subordinate to this cover. 
	We decompose $\phi = \sum_{i=1}^m \phi_i$, where each $\phi_i := \rho_i \phi$ is compactly supported in $U_i$. 
	For each $i$, let $G_i : Q \to U_i$ be a smooth coordinate function from the open unit  cube $Q \subset \mathbb{R}^N$ to $U_i$. The  pullback  $G_i^* \phi_i$  belongs to  $W^{1,N}\left(Q, \bigwedge^l T^* Q\right)$ and has  compact support in $Q$, hence can be viewed as an element of $W_0^{1,N}\left(Q, \bigwedge^\ell T^* Q\right)$.

	By \cite[Theorem $5''$]{BB}, there exists  $\tilde{\psi}_i \in \left(W_0^{1,N} \cap L^\infty\right)\left(Q, \bigwedge^\ell T^* Q\right)$ such that
	$$
	\dd\tilde{\psi}_i = \dd\left(G_i^* \phi_i\right) \quad \text{on } Q,
	$$
	with the scale-invariant bound 
	$$
	\|\nabla\tilde{\psi}_i\|_{L^N(Q)} + \|\tilde{\psi}_i\|_{L^\infty(Q)} \le C \|\dd(G_i^* \phi_i)\|_{L^N(Q)} \le C \|G_i^* \phi_i\|_{W^{1,N}(Q)}.
	$$ 
	%(i.e., $d \circ G_i^* = G_i^* \circ d$)

	Since the exterior derivative commutes with smooth pullback, \emph{i.e.}, $\dd \circ G_i^* = G_i^* \circ \dd$, pushing forward via $(G_i^{-1})^*$ yields a form $\psi_i = (G_i^{-1})^* \tilde{\psi}_i$ on $U_i$ satisfying $\dd\psi_i = \dd\phi_i$. 
	Since  $\tilde{\psi}_i \in W^{1,N}_0(Q)$, its pushforward has vanishing trace on $\partial U_i$, and  can be extended by zero to a global $W^{1,N}$-form on $X$. Extending each $\psi_i$ by zero to  $X$ and defining $\psi := \sum_{i=1}^m \psi_i$, the linearity of $\dd$  gives
	$$
\dd\psi = \sum_{i=1}^m \dd\psi_i = \sum_{i=1}^m \dd\phi_i = \dd\phi \qquad \text{on } X.
	$$
	Summing the local estimates yields $\psi \in \left(W^{1,N}\cap L^\infty\right)\left(X,\bigwedge^\ell T^*X\right)$ with
	 $$\|\psi\|_{W^{1,N}(X)} + \|\psi\|_{L^\infty(X)} \le C \|\phi\|_{W^{1,N}(X)}.$$
	
	Set $\omega := \phi - \psi \in W^{1,N}\left(X,\bigwedge^\ell T^*X\right)$. By construction, $\omega$ is a  closed $\ell$-form on $X$. 
	
	By the $L^p$-Hodge decomposition on closed manifolds for $1<p<\infty$, one has
	$$
	L^p\left(X,\bigwedge^\ell T^*X\right) =  \dd\left[W^{1,p}\left(X,\bigwedge^{\ell-1}T^*X\right)\right] \bigoplus \dd^*\left[W^{1,p}\left(X,\bigwedge^{\ell+1}T^*X\right)\right] \bigoplus\mathcal{H}^\ell(X).
	$$
	Applying this decomposition to $\omega$ with $p=N$, and noting that $\dd\omega = 0$ implies its co-exact projection vanishes, we deduce that $\omega = \dd\eta + h$
	 for some $h \in \mathcal{H}^l(X)$. 
	Standard elliptic regularity for the Hodge Laplacian on $W^{1,N}(X)$ implies	that $\eta$ can be chosen in $W^{2,N}\left(X, \bigwedge^{\ell-1}T^*X\right)$ with  $\|\eta\|_{W^{2,N}(X)} \le C \|\omega\|_{W^{1,N}(X)}$. Since $\mathcal{H}^\ell$ is a finite-dimensional space consisting of smooth forms, 
we have	$$\|h\|_{L^\infty(X)} \le C \|h\|_{L^2(X)} \le C \|\omega\|_{W^{1,N}(X)}.$$
	
	Substituting $\omega = \dd\eta + h$ into $\phi = \psi + \omega$, we obtain the desired decomposition $\phi = \psi + \dd\eta + h$. Finally, combining the estimates for $\eta$ and $h$ with the triangle inequality, we obtain the global estimate \eqref{eq:HBB-estimate}, namely that $$\|\omega\|_{W^{1,N}(X)} \le \|\phi\|_{W^{1,N}(X)} + \|\psi\|_{W^{1,N}(X)} \le C\|\phi\|_{W^{1,N}(X)}.$$ All the constants $C$ depend only on $\ell$ and the geometry of the manifold $X$.
\end{proof}

To proceed, let \(\{h_j\}\) be an \(L^2\)-orthonormal basis of  \(\mathcal{H}^l\).  
Define  
\[
h_\xi := \sum_j \langle \xi, h_j \rangle h_j.
\]  
%where \(\langle \xi, h_j \rangle := \int_X \langle \xi, h_j \rangle_g dV_g\).   Since each \(h_j\) is smooth and bounded, \(\langle \xi, h_j \rangle\) is well-defined. 
%by \(L^1\)-\(L^\infty\) duality.  
Then H\"older's inequality yields that
$$
\|h_\xi\|_{L^\infty(X)} \le \sum_j \|\xi\|_{L^1(X)} \|h_j\|_{L^\infty(X)}^2 \le C \|\xi\|_{L^1(X)}.
$$
%Since $\mathcal H^l$ is finite-dimensional, all norms on it are equivalent, and hence
Hence, one has
\begin{equation}
\|h_\xi\|_{W^{-1,N'}(X)} \leq C \|h_\xi\|_{L^\infty(X)} \leq C \|\xi\|_{L^1(X)}. \label{eq:harmonic_bound}
\end{equation}

As a result,
$$\tilde\xi := \xi - h_\xi$$
is orthogonal to \(\mathcal{H}^\ell\) in the distributional sense. %Indeed, one has
%\begin{align*}
%\langle \tilde\xi, h_k \rangle&=\langle \xi, h_k \rangle-\langle h_\xi, h_k \rangle\\
%&=\langle \xi, h_k \rangle-\langle {\textstyle\sum}_j \langle \xi, h_j \rangle h_j, h_k \rangle\\
%&=\langle \xi, h_k \rangle-{\textstyle\sum}_j \langle \xi, h_j \rangle\langle  h_j, h_k \rangle\\
%&=\langle \xi, h_k \rangle-{\textstyle\sum}_j \langle \xi, h_j \rangle\delta_{jk}\\
%&=\langle \xi, h_k \rangle-\langle \xi, h_k \rangle=0.
%\end{align*}
 We next show that  $\tilde{\xi} \in W^{-1,N'}\left(X, \bigwedge^\ell T^*X\right)$.

\begin{lemma}\label{L2}
Let $\phi\in W^{1,N}\left(X,\bigwedge^\ell T^*X\right)$ be as above. Then 
	$\tilde{\xi} \in W^{-1,N'}\left(X, \bigwedge^\ell T^*X\right)$ with the estimate
	\begin{equation}\label{eq:rhs_bound}
	\|\tilde{\xi}\|_{W^{-1,N'}(X)}  \leq C \big( \|\xi\|_{L^1(X)} + \|\dd^*\xi\|_{W^{-2,N'}(X)} \big). 
	\end{equation}
\end{lemma}
\begin{proof}
By Lemma \ref{L1} we have the decomposition  
\[
\phi = \psi + \dd\eta + h_\phi,
\]  
with  
\[
\psi \in \left(W^{1,N}\cap L^\infty\right)\left(X,\bigwedge^\ell T^*X\right),\quad 
\eta \in W^{2,N}\left(X,\bigwedge^{\ell-1}T^*X\right),\quad 
h_\phi\in \mathcal{H}^\ell(X)
\]  
satisfying the uniform bound  
\[
\|\psi\|_{W^{1,N}(X)} + \|\psi\|_{L^\infty(X)} + \|\eta\|_{W^{2,N}(X)} + \|h_\phi\|_{L^\infty(X)}
\le C\|\phi\|_{W^{1,N}(X)}.
\]
Then  
\[
\langle \tilde\xi, \phi \rangle
= \langle \tilde\xi, \psi \rangle + \langle \tilde\xi, \dd\eta \rangle + \langle \tilde\xi, h_\phi \rangle.
\]  
Thanks to H\"older's inequality, the first and the third terms on the right-hand side are estimated as follows:
\[
|\langle \tilde\xi, \psi \rangle| \le \|\tilde\xi\|_{L^1(X)}\|\psi\|_{L^\infty(X)}
\le C\|\xi\|_{L^1(X)}\|\phi\|_{W^{1,N}(X)},
\]  
\[
|\langle \tilde\xi, h_\phi \rangle| \le \|\tilde\xi\|_{L^1(X)}\|h_\phi\|_{L^\infty(X)}
\le C\|\xi\|_{L^1(X)}\|\phi\|_{W^{1,N}(X)}.
\]
For the middle term, note that \(d^*\tilde\xi = d^*\xi\) since \(d^*h_\xi = 0\).
% (harmonic forms are closed and co-closed).  
Hence, using the distributional definition of \(\dd^*\), we have that 
\[
\langle \tilde\xi, \dd\eta \rangle = \langle \dd^*\tilde\xi, \eta \rangle
= \langle \dd^*\xi, \eta \rangle.
\]  
Since \(\dd^*\xi \in W^{-2,N'}(X)\) and \(\eta \in W^{2,N}(X)\), duality gives  
\[
|\langle \dd^*\xi, \eta \rangle| \le \|\dd^*\xi\|_{W^{-2,N'}(X)} \|\eta\|_{W^{2,N}(X)}
\le C \|\dd^*\xi\|_{W^{-2,N'}(X)} \|\phi\|_{W^{1,N}(X)}.
\]
Collecting the estimates, we obtain  
\[
|\langle \tilde\xi, \phi \rangle|
\le C\bigl( \|\xi\|_{L^1(X)} + \|\dd^*\xi\|_{W^{-2,N'}(X)} \bigr) \|\phi\|_{W^{1,N}(X)},
\]  
which ends the proof.
\end{proof}

%By the density of $(W^{1,N} \cap L^\infty)(X, \wedge^l T^*X)$ in $W^{1,N}(X, \wedge^l T^*X)$, the estimate in Lemma \ref{L2} implies that the linear functional $\phi \mapsto \langle \tilde{\xi}, \phi \rangle$ is bounded on $W^{1,N}(X, \wedge^l T^*X)$. By the Hahn-Banach theorem, it  extends uniquely to a bounded linear functional 

We are now ready to conclude Lemma~\ref{lem: brezis-vs}.

\begin{proof}[Proof of  Lemma~\ref{lem: brezis-vs}]
 By elliptic regularity theory for the Laplace--Beltrami operator on closed manifolds,  
\[
\Delta : W^{1,N'}\left(X,\bigwedge^\ell T^*X\right) \cap \left(\mathcal{H}^\ell(X)\right)^\perp
\longrightarrow W^{-1,N'}\left(X,\bigwedge^\ell T^*X\right) \cap \left(\mathcal{H}^\ell(X)\right)^\perp
\]  
is an isomorphism.  
Since $\tilde\xi$ belongs to the space on the right-hand side, there exists a unique  
\(\sigma \in W^{1,N'}\left(X,\bigwedge^\ell T^*X\right) \cap \left(\mathcal{H}^\ell(X)\right)^\perp\) such that  
\[
\Delta \sigma = \tilde\xi,
\]  
with the estimate  
\[
\|\sigma\|_{W^{1,N'}(X)} \le C \|\tilde\xi\|_{W^{-1,N'}(X)}.
\]
Since \(\Delta\sigma = \xi - h_\xi\), we have solved \(\Delta\sigma = \xi\) modulo harmonic forms.  
Finally, note that  
\[
\inf_{h\in\mathcal{H}^\ell(X)} \|\dd\sigma + h\|_{L^{N'}(X)}
\le \|\dd\sigma\|_{L^{N'}(X)} \le \|\sigma\|_{W^{1,N'}(X)},
\]  
which together with \eqref{eq:rhs_bound} yields the desired \emph{a priori} estimate~\eqref{estimate in Lemma B}.

For the case that $\xi \in \mathcal{M}\left(X, \bigwedge^\ell T^* X\right)$ is a finite Radon measure, the proof follows from a standard approximation argument. One may regularise  $\xi$ by a sequence of smooth forms \(\{\xi_k\}\) converging to \(\xi\) in the weak-\(*\) topology of measures, such that $\|\xi_k\|_{L^{1}(X)}\leq C\|\xi\|_{\mathcal{M}(X)}$. Applying the above arguments to each $\xi_k$ yields $\sigma_k\in W^{1,N'}\left(X,\bigwedge^\ell T^*X\right)$ such that $\Delta\sigma_k = \xi_k - h_{\xi_k}$ with $\|\sigma_k\|_{W^{1,N'}(X)}$ uniformly bounded. A subsequence $\{\sigma_{k_j}\} \subset \{\sigma_k\}$ converges weakly to some \(\sigma\) in \( W^{1,N'}(X)\).
Using the weak-$\ast$ convergence of $\xi_k$ to $\xi$ and the boundedness of $\{h_{\xi_k}\}$ in the finite-dimensional space $\mathcal{H}^{\ell}(X)$, we may select a further subsequence so that $h_{\xi_k}\to h_{\xi}\in\mathcal{H}^{\ell}(X)$.
Passing to the limit gives $\Delta\sigma = \xi - h_{\xi}$, i.e. $\Delta\sigma\equiv\xi\pmod{\mathcal{H}^{\ell}}$.
Finally, the desired estimate follows from lower semicontinuity and \eqref{estimate in Lemma B}.   \end{proof}

\begin{remark}
In passing, we comment that the assumption $\varphi \in \left( W^{1,N}_0\cap L^\infty\right)(\Omega;\mathbb{R}^N)$ in \cite[Lemma 3.4]{bv} can be relaxed to $\varphi \in W^{1,N}_0(\Omega;\mathbb{R}^N)$.  Indeed, by \cite[Lemma 3.3]{bv},  any   $\varphi\in W^{1,N}_0(\Omega;\mathbb{R}^N)$ admits a decomposition $\varphi = \psi + \nabla\eta$ with $\psi \in \left(W^{1,N}_0\cap L^\infty\right)(\Omega;\mathbb{R}^N)$ and $\eta\in W_0^{2,N}(\Omega)$. The term $\int_\Omega f \cdot \psi$ is well defined since $f \in L^1(\Omega;\mathbb{R}^N)$ and $\psi \in L^\infty(\Omega;\mathbb{R}^N)$,  while$\int_\Omega f \cdot \nabla\eta$ is controlled by $[f]$ (see \cite{bv} for notation). Thus there is no need to assume  $\varphi \in L^\infty$.
\end{remark}

\bigskip
\noindent
{\bf Acknowledgement}.
Siran Li is indebted to Professor Armin Schikorra, from whom I learned a lot about  compensated compactness over time. I also thank New York University-Shanghai for providing excellent working atmosphere when I undertook an adjunct professorship and visiting scholarship there, during which the major part of this work was completed.

 The research of SL is supported by NSFC Projects 12201399, 12331008, and 12411530065, Young Elite Scientists Sponsorship Program by CAST 2023QNRC001, the National Key Research $\&$ Development Program 2023YFA1010900 and 2024YFA1014900,  Shanghai Rising-Star Program 24QA2703600,  and the Shanghai Frontier Research Institute for Modern Analysis. The research of XS is partially supported by the National Key Research $\&$ Development Programs 2023YFA1010900  and 2024YFA1014900.

\end{document}